\documentclass[preprint]{elsarticle}

\usepackage{amsmath}
\usepackage{amsfonts}
\usepackage{amssymb}
\usepackage{amsthm}		
\usepackage{eucal}		
\usepackage{bm}                 
\usepackage{mathtools}          

\usepackage{hyperref}
\newcommand{\oeisseqnum}[1]{\href{https://oeis.org/#1}{\rm \underline{#1}}}

\renewcommand\ge\geqslant
\renewcommand\le\leqslant

\newcommand{\Hermiteprob}{\mathop{{}\it He}\nolimits}

\newtheorem{theorem}{Theorem}[section]
\newtheorem{corollary}[theorem]{Corollary}
\newtheorem{proposition}[theorem]{Proposition}
\newtheorem{lemma}[theorem]{Lemma}
\newtheorem{conjecture}[theorem]{Conjecture}

\theoremstyle{definition}

\theoremstyle{remark}

\newtheorem{remarkaftertheorem}{Remark}[theorem]

\newcommand{\adag}{{a^\dag}}

\newcommand{\defeq}{\mathrel{:=}}
\newcommand{\eqdef}{\mathrel{=:}}

\newcommand\stircyc[2]{\genfrac{[}{]}{0pt}{}{#1}{#2}}
\newcommand\stirsub[2]{\genfrac{\{}{\}}{0pt}{}{#1}{#2}}
\newcommand\eulerian[2]{\genfrac\langle\rangle{0pt}{}{#1}{#2}}

\newenvironment{smallarray}[1]
 {\null\,\vcenter\bgroup\scriptsize
  \arraycolsep=.13885em
  \hbox\bgroup$\array{@{}#1@{}}}
 {\endarray$\egroup\egroup\,\null}



\begin{document}


\begin{frontmatter}
\title{Boson operator ordering identities from generalized Stirling and Eulerian numbers}
\author{Robert S. Maier}
\ead{rsm@math.arizona.edu}
\address{Depts.\ of Mathematics and Physics, University of Arizona, Tucson,
AZ 85721, USA}

\begin{abstract}
Ordering identities in the Weyl--Heisenberg algebra generated by
single-mode boson operators are investigated.  A~boson string composed
of creation and annihilation operators can be expanded as a linear
combination of other such strings, the simplest example being a normal
ordering.  The case when each string contains only one annihilation
operator is already combinatorially nontrivial.  Two kinds of
expansion are derived: (i)~that of a power of a string~$\Omega$ in
lower powers of another string~$\Omega'$, and (ii)~that of a power
of~$\Omega$ in twisted versions of the same power of~$\Omega'$.  The
expansion coefficients are shown to be, respectively, generalized
Stirling numbers of Hsu and Shiue, and certain generalized Eulerian
numbers.  Many examples are given.  These combinatorial numbers are
binomial transforms of each other, and their theory is developed,
emphasizing schemes for computing them: summation formulas,
Graham--Knuth--Patashnik (GKP) triangular recurrences, terminating
hypergeometric series, and closed-form expressions.  The results on
the first type of expansion subsume a number of previous results on
the normal ordering of boson strings.
\end{abstract}
\end{frontmatter}


\section{Introduction}
\label{sec:intro}

Single-mode boson creation and annihilation operators $a^\dag,a$ do
not commute, but satisfy the commutation relation $[a,a^\dag]=1$.
Over the complex numbers, they generate the Weyl--Heisenberg
algebra~$W_{\mathbb{C}}$.  Similarly noncommuting pairs of operators
appear in many places in quantum mechanics and quantum field theory,
and many examples could be adduced~\cite{Fleury94}.  For instance, the
position operator~$x$ and the operator $D=D_x={\rm d}/{\rm d}x$
satisfy the equivalent relation $[D,x]=1$.  They too generate a
$\mathbb{C}$\nobreakdash-algebra, but negative integer powers of~$x$
and indeed arbitrary powers of~$x$ will be allowed here, so the
algebra is a larger one,~$\tilde W_{\mathbb{C}}$.  The algebra
$W_{\mathbb{C}}$~has a representation on Fock space, and $\tilde
W_{\mathbb{C}}$~can be represented on any convenient space of
functions of~$x$, but this paper will deal with their algebraic
properties rather than their representations.

Reordering the operators in any algebra element built from~$a^\dag,a$
(or~$x,D$), to satisfy constraints, may be combinatorially nontrivial.
A~standard problem is that of normal ordering: moving all
operators~$a$ (or~$D$) to the right, to facilitate the computing of
quantum expectation values.  It has long been known~\cite{Katriel74}
that the normal ordering of $(a^\dag a)^n$ and $(a \adag)^n$ involves
Stirling numbers of the second kind~\cite{Graham94}, which are denoted
by $S(n,k)$ or~$\stirsub{n}k$.  In present notation,
\begin{equation}
\label{eq:katriel}
(\adag a)^n = \sum_{k=0}^n S_{n,k}(0,1;0)\, \adag^k (a)^k,  
\qquad 
(a \adag)^n = \sum_{k=0}^n S_{n,k}(0,1;1)\, \adag^k (a)^k,
\end{equation}
where the coefficients $S_{n,k}(0,1;0)$ and $S_{n,k}(0,1;1)$
respectively equal $\stirsub{n}k$ and $\stirsub{n+1}{k+1}$.  These are
\emph{ordering identities}, and they have been greatly generalized
\cite{Mansour2016,Schork2021}.  For example, if the word raised to the
$n$'th power is an arbitrary word $\Omega=\Omega(a^\dag,a)$ in the
noncommuting~$a^\dag,a$, or a linear combination, it~may be possible
to compute the coefficients in the normal ordering of~$\Omega^n$
iteratively or recursively, or view them as counting combinatorial
structures~\cite{Blasiak2012}.

This paper derives ordering identities in the Weyl--Heisenberg algebra
of two new kinds, not previously considered in detail.  The derivation
is largely algebraic rather than combinatorial.  In the first, one
expands the $n$'th power of a word $\Omega=\Omega(a^\dag, a)$ in
powers $\Omega{}^\prime{}^k$, $k=0,\dots,n$, of another specified
word~$\Omega'$.  For simplicity, each of $\Omega,\Omega'$ will be
required to contain exactly one annihilation operator~$a$, so $\Omega
= \adag^L a\adag^R$ and $\Omega' = \adag^{L'} a\adag^{R'}$, or
correspondingly, $\Omega = x^L D x^R$ and $\Omega'= x^{L'} D x^{R'}$.
The first theorem dealing with this situation contains the following
parametric operator identity in~$\tilde W_{\mathbb{C}}$ (see
Theorem~\ref{thm:firstmain}):
\begin{equation}
\begin{aligned}
  &\left(x^{1-L-R}\right)^n \left(x^L D x^R\right)^n\\
  &\qquad {}= \sum_{k=0}^n S_{n,k}\left(1-L-R,1-L'-R';\,R-R'\right)
  \left(x^{1-L'-R'}\right)^k \left(x^{L'}Dx^{R'}\right)^k.
\label{eq:firstmain}
\end{aligned}
\end{equation}
Here, $L,R,L',R'\in\mathbb{C}$ are arbitrary, and the coefficients
$S_{n,k}$, $0\le k\le n$, with three complex parameters, are
``unified'' generalized Stirling numbers of Hsu and
Shiue~\cite{Hsu98}.  They are one among many generalizations of the
classical Stirling numbers; for others, see
\cite{Mansour2016,Sandor2004}.

One can formally substitute $\adag,a$ for~$x,D$
in~(\ref{eq:firstmain}), but for the result to be an identity in
$W_{\mathbb{C}}$, it must contain only nonnegative integer powers of
$\adag$ and~$a$.  So one must require $L,R,L',R'\in\mathbb{N}$, and to
remove negative powers, left-multiply both sides
of~(\ref{eq:firstmain}) by~$(x^E)^n$, where $E\in\mathbb{N}$~is
sufficiently large.  With $x,D$ replaced by~$a^\dag,a$, this yields a
parametric identity in~$W_{\mathbb{C}}$:
\begin{equation}
\begin{aligned}
  \left({a^\dag}^{1-L-R+E}\right)^n \left({a^\dag}^L a {a^\dag}^R\right)^n
  =& \sum_{k=0}^n S_{n,k}\left(1-L-R,1-L'-R';\,R-R'\right)\\
   &\quad{}\times
  \left({a^\dag}^E\right)^{n-k}   \left({a^\dag}^{1-L'-R'+E}\right)^k \left({a^\dag}^{L'}a{a^\dag}^{R'}\right)^k,
\label{eq:firstmainmod}
\end{aligned}
\end{equation}
where taking $E=\max\left(L+R-1,L'+R'-1,0\right)$ suffices (see
Theorem~\ref{thm:powerful}).

The normal orderings in~(\ref{eq:katriel}) are the
$(L,R;L',R')=(1,0;0,0)$ and $(0,1;0,0)$ cases
of~(\ref{eq:firstmainmod}), with $E=0$, and several other known
identities \cite{Kim2022,Lang2000,Mikhailov85,MohammadNoori2011} are
also subsumed by~(\ref{eq:firstmainmod}).  This includes the normal
ordering of any monomial $({a^\dag}^L a{a^\dag}^R)^n$, which comes by
taking $(L',R')=(0,0)$ and $\Omega'=\nobreak a$.  It is known that
normally ordering~$\Omega^n$ when $\Omega$~is a ``single annihilator''
word yields generalizations of the classical Stirling
numbers~\cite{DuChamp2004,DuChamp2010}, and the case when
$\Omega^n$~is replaced by a product $\adag^{r_m}a^{s_m}\dots
\adag^{r_1}a^{s_1}$, which includes $({a^\dag}^L a {a^\dag}^R)^n$, has
been treated~\cite{Mendez2005}. But the
expansion~(\ref{eq:firstmainmod}) and the extension to words
$\Omega'=\adag^{L'} a \adag^{R'}$ other than~$a$ are new.

It is explained at length below how the Hsu--Shiue Stirling numbers
$S_{n,k}$ appearing as coefficients can be computed: by summation
formulas, triangular recurrences, or identities relating them at
different indices~$n,k$ and different values of their parameters.
Closed-form expressions at certain parameter values are also derived.
A~sample application of eq.~(\ref{eq:firstmainmod}) is
\begin{equation}
\label{eq:sampleappl}
\begin{aligned}
  \left( {a^\dag}^3 a\right)^n  &=  {a^\dag}^n\,\sum_{k=0}^n
  S_{n,k}(-2,-1;0)\:
{a^\dag}^{(n-k)} \left({a^\dag}^2 a\right)^k\\
&=  {a^\dag}^n \, \sum_{k=0}^n 2^{-(n-k)}k\,\binom{n}{k} \frac{(2n-k-1)!}{n!}\:  {a^\dag}^{(n-k)} \left({a^\dag}^2 a\right)^k, \qquad n\ge1,
\end{aligned}
\end{equation}
which does not normally-order $\Omega^n=({a^\dag}^3a)^n$, but rather
expands it in powers of $\Omega'={a^\dag}^2a$.  Here,
$(L,R;L',R')=(3,0;2,0)$ and $E=2$.  This formula incorporates a
closed-form expression for the numbers $S_{n,k}(-2,-1;0)$ with ${0\le
  k\le n}$, which applies when $n\ge1$.  They are also called the
unsigned Bessel numbers of the first kind, and are denoted below
by~$\hat b_{n,k}$.

There has been some work \cite{Kim2023,Mansour2011,Mansour2012a} on
the combinatorics of the algebra $W_{\mathbb{C}}(s)$ over~$\mathbb{C}$
defined by $UV-\nobreak VU=V^s$.  When $s=0$, this
is~$W_{\mathbb{C}}$, and when $s\in\mathbb{N}$ it~can be identified
with a subalgebra of~$W_{\mathbb{C}}$, its embedding
in~$W_{\mathbb{C}}$ being given by $V=\adag$, $U=\adag^s a$.
In~particular, $W_{\mathbb{C}}(1)$ is the well-known shift algebra,
and $W_{\mathbb{C}}(2)$ is the Jordan plane, also called the
meromorphic Weyl algebra~\cite[Section~7.3]{Mansour2016}.  Many
ordering identities in~$W_{\mathbb{C}}(s)$ follow from the parametric
identity~(\ref{eq:firstmainmod}), and the appearance of Hsu--Shiue
generalized Stirling numbers in the context of $W_{\mathbb{C}}(s)$ has
already been noted~\cite{Mansour2012}.

\smallskip
The second kind of ordering identity derived in this paper is based on
a generalization of the classical Eulerian numbers, which are
traditionally denoted by~$A_{n,k}$ and have more recently been denoted
by $\eulerian{n}k = A_{n,k+1}$.  These numbers are not well
known to physicists, their combinatorial interpretation having been
discovered as late as the 1950's \cite{Comtet74,Graham94,Riordan58}.
Identities in the Weyl--Heisenberg algebra~$W_{\mathbb{C}}$ based on
the classical numbers are (in~present notation)
\begin{subequations}
\label{eq:euleriank}
\begin{align}
  n!\,(\adag a)^n &= \sum_{k=0}^n E_{n,k}(0,1;\,0)\, \adag^k (a)^n \adag^{(n-k)},\label{eq:euleriank1}\\
n!\,(a \adag)^n &= \sum_{k=0}^n E_{n,k}(0,1;\,1)\, \adag^k (a)^n \adag^{(n-k)},\label{eq:euleriank2}
\end{align}
\end{subequations}
where the coefficients $E_{n,k}(0,1;0)$ and $E_{n,k}(0,1;1)$
respectively equal $A_{n,k}$ and $\eulerian{n}k=A_{n,k+1}$.

Generalized Eulerian numbers $E_{n,k}$, $0\le k\le n$, with three
complex parameters, which are of the type needed for
extending~(\ref{eq:euleriank}), have been considered in
\cite{Carlitz79,Charalambides82,Hsu99} and by the
author~\cite{Maier28}; for other generalizations,
see~\cite{Sandor2004}.  These numbers appear as coefficients in
expansions that are generalized versions of (\ref{eq:euleriank1})
and~(\ref{eq:euleriank2}).  Given any single-annihilator word
$\Omega=\adag^L a \adag^R $, one can expand its $n$'th power in
``twistings'' of the $n$'th power of another such word,
$\Omega'=\adag^{L'} a \adag^{R'}$.  More precisely, for sufficiently
large $p,q\in\mathbb{N}$, one can expand
$\left({a^\dag}^p\right)^n\Omega{}^n\left({a^\dag}^q\right)^n$ as a
combination of terms of the form $\adag^{p'} {\Omega'}{}^n
\adag^{q'}$, with each term having the same value of~$p'+\nobreak q'$,
which is a certain multiple of~$n$.  In each identity
in~(\ref{eq:euleriank}), $p=q=0$, with $\Omega'=a$ and $p'+q'=n$.

A general expansion of this type (the first of four in
Theorem~\ref{thm:powerful2}) is
\begin{equation}
\label{eq:introE2}
\begin{aligned}
  &n!\, (-e')^n \left({a^\dag}^{(E_L-e)}\right)^n \left({a^\dag}^L a
  {a^\dag}^R\right)^n \left(\adag^{E_R}\right)^n
  \\ &\qquad{}=\sum_{k=0}^n E_{n,k}\left(-e,-e';\,R-R'\right)
  \\ &\qquad\qquad{}\times\left({a^\dag}^{E_L}\right)^{n-k}
  \left({a^\dag}^{({E_L}-e')}\right)^k
  \left({a^\dag}^{L'}a{a^\dag}^{R'}\right)^n
  \left(\adag^{(E_R-e')}\right)^{n-k} \left(\adag^{E_R}\right)^k,
\end{aligned}
\end{equation}
where $e=L+R-1$, $e'=L'+R'-1$, with $E_L=\max(e,e',0)$ and
$E_R=\max(e',0)$.  In~(\ref{eq:introE2}), one sees that
$p=E_L-\nobreak e$, $q=E_R$, $p'+\nobreak q' = (E_L+\nobreak
E_R-\nobreak e')n$.  Similarly to~(\ref{eq:firstmainmod}), this
expansion is obtained from a simpler one in~$\tilde W_{\mathbb{C}}$
(see Theorem~\ref{thm:secondmain}) by left-multiplying
by~$(x^{E_L})^n$ and right-multiplying by~$(x^{E_R})^n$, with
$E_L,E_R$ sufficiently large that each negative power of~$x$ on either
side is removed or made positive.  The identities
(\ref{eq:euleriank1}),(\ref{eq:euleriank2}) come
from~(\ref{eq:introE2}) by setting $(L,R;\allowbreak
L',R')=(1,0;\allowbreak 0,0)$ and~$(0,1;0,0)$ respectively, so that
(in~each case) $e=\nobreak0$, $e'=\nobreak-1$, and $E_L=\nobreak
E_R=\nobreak0$.

As with the $S_{n,k}$, the generalized Eulerian numbers $E_{n,k}$
appearing as coefficients can be computed by various means: summation
formulas, triangular recurrences, or identities relating them at
different indices $n,k$ and different values of their parameters.
Closed-form expressions at certain parameter values are also derived
below.  Another sample application of~(\ref{eq:introE2}) is the
expansion
\begin{equation}
\label{eq:sampleeulerian}
  2^n\,(\adag)^n (\adag^2 a)^n (\adag^2)^n = \sum_{k=0}^n
  \binom{n+1}{2k+1}\, (\adag^2)^k (\adag^3 a)^n (\adag^2)^{n-k},
\end{equation}
in which $\Omega^n = (\adag^2 a)^n $ and $\Omega'{}^n = (\adag^3
a)^n$.  This comes by setting $(L,R;L',R')=(2,0;3,0)$, so that $e=1$,
$e'=2$, $E_L=E_R=2$, and replacing the summation index~$k$
by~$n-\nobreak k$.  The binomial coefficients on the right-hand side
come from the closed-form expression $(-1)^n\, n!\binom{n+1}{2n-2k+1}$
for $E_{n,k}(-1,-2;0)$, and the only nonvanishing terms in the sum are
those with $0\le k\le\lfloor\frac{n}2\rfloor$.

\smallskip
This paper is heavily influenced by the early work of Toscano, who in
a series of papers published in the 1930's (mostly cited
in~\cite{Mansour2016}) obtained many singly parametrized differential
operator expansions, the coefficients in which satisfy triangular
recurrences.  Among his tools was a generalized Boole identity
(Lemma~\ref{lem:major}), which will play a role below.  An entry point
to his work is provided by~\cite{Toscano39}.  Other than being
$1$\nobreakdash-indexed rather than $0$\nobreakdash-indexed, most of
his expansion coefficients are one-parameter specializations of the
three-parameter $S_{n,k}$ and~$E_{n,k}$ used here.  For instance, his
$B_{n+1,k+1}^{(u)}$ equals $E_{n,k}(1-\nobreak u, 1; u)$.

The plan of the paper is as follows.  Section~\ref{sec:label} explains
notation and mathematical conventions.  Section~\ref{sec:algebra}
contains some basic facts about the algebras $W_{\mathbb{C}}$
and~$\tilde W_{\mathbb{C}}$, including the generalized Boole identity.
In Section~\ref{sec:fundamentals} the parametric numbers $S_{n,k}$
and~$E_{n,k}$ are introduced and investigated in depth, with an
emphasis on recurrences, summation formulas, and other computationally
useful results, such as the fact that they are binomial transforms of
each other.  Section~\ref{sec:stirlingformulas} derives some
closed-form expressions for the~$S_{n,k}$ and comments on
combinatorial interpretations.  Ordering identities of the generalized
Stirling kind are derived in Section~\ref{sec:stirlingids}\null.
Section~\ref{sec:eulerianformulas}~derives some closed-form
expressions for the~$E_{n,k}$ and briefly discusses combinatorial
interpretations.  Finally, ordering identities of the generalized
Eulerian kind are derived in Section~\ref{sec:eulerianids}.

\section{Notation and conventions}
\label{sec:label}

Knuth's overline and underline notations~\cite{Knuth92} are used for
rising and falling factorial powers.  For all $r\in\mathbb{C}$
and~$m\in\mathbb{N}$, $(r)^{\overline m}$~equals
$(r)(r+\nobreak1)\dots\allowbreak(r+\nobreak m-\nobreak1)$, with
$(r)^{\overline 0}=1$.  Likewise, $(r)^{\underline m}$~equals
$(r)(r-\nobreak1)\dots\allowbreak(r-\nobreak m+\nobreak1)$, with
$(r)^{\underline 0}=1$.

These notations are extended here to include a non-unit ``stride.''
For all $\alpha\in\mathbb{C}$, $(r)^{\overline m,\alpha}$~equals $(r)(r+\nobreak
\alpha)\dots\allowbreak[r+\nobreak (m-\nobreak 1)\alpha]$, and likewise,
$(r)^{\underline m,\alpha}$~equals $(r)(r-\nobreak
\alpha)\dots\allowbreak[r-\nobreak (m+\nobreak 1)\alpha]$.  Note also that
$\binom{n}k$ is defined for all $n,k\in\mathbb{N}$, with
$\binom{n}k=0$ if~$k>n$, and that~$0^0=1$.

Knuth's notations $\stirsub{n}k$, $\stircyc{n}k$, and $\eulerian{n}k$,
where $0\le k\le n$, are used for the classical Stirling numbers of
the second kind, the unsigned ones of the first kind, and the Eulerian
numbers. The last differs from the classical definition in having its
lower index shifted by unity.  That is, $\eulerian{n}k=A_{n,k+1}$.

The three-parameter generalized Eulerian numbers $E_{n,k}(\alpha,\beta;r)$
employed here are specializations of four-parameter ones of the
author.  In previously used notation and terminology~\cite{Maier28},
the $E_{n,k}(\alpha,\beta;r)$ would be written as $E_{n,k}(\alpha,\beta;\allowbreak
r,\beta-\nobreak r)$ and called ``single progression'' generalized
Eulerian numbers.

The triangular recurrences derived below, satisfied by $E_{n,k}$
and~$S_{n,k}$, are of the GKP type, $c_{n+1,k+1}=[\alpha n+\beta (k+1)
  + \gamma]\,c_{n,k+1} + [\alpha' n+\beta'k + \gamma']\,c_{n,k}$.  The
importance of such recurrences was noted by Graham, Knuth, and
Patashnik \cite[Chapter~6,\ Problem~94]{Graham94}, after whom they are
named.

Several products $\mathcal{A}\mathcal{B}$ of infinite matrices appear,
the rows and columns of each matrix being indexed by~$\mathbb{N}$.
In~all cases $\mathcal{A}$~is row-finite (each of its rows has finite
support), or $\mathcal{B}$~is column-finite.  Hence each element
of~$\mathcal{A}\mathcal{B}$ can be computed as a finite sum.

In a standard definition, terminating hypergeometric series are those
in which an ``upper'' parameter is a nonpositive integer, such as the
Gauss hypergeometric series
\begin{equation}
  {}_2F_1
  \biggl[
    \begin{array}{cc}
      -N, & b
      \\
      & c
    \end{array}
    \biggm|
    z\,
    \biggr]
  \defeq
  \sum_{k=0}^\infty \frac{(-N)^{\overline k}(b)^{\overline k}}{(c)^{\overline k}}\, \frac{z^k}{k!}\,.
\end{equation}
If $N\in\mathbb{N}$, one expects the series to terminate with the
$k=N$ term.  Usually the lower parameter~$c$ is required not to be a
nonpositive integer, as the series is not defined due~to division(s)
by zero if $c$~is an integer in the range ${-N\le c\le 0}$.  To~cover
such cases, a desingularized function~${}_2\hat F_1$, defined
informally as $(c)^{\overline N}\,{}_2F_1$, is used here.  That is,
\begin{equation}
  {}_2\hat F_1
  \biggl[
    \begin{array}{cc}
      -N, & b
      \\
      & c
    \end{array}
    \biggm|
    z\,
    \biggr]
  \defeq
  \sum_{k=0}^N (-N)^{\overline k}(b)^{\overline k} (c+k)^{\overline {N-k}} \: \frac{z^k}{k!}\,,
\end{equation}
because $(c)^{\overline N}/(c)^{\overline k}=(c+k)^{\overline{N-k}}$
if there is no~division by zero.  ${}_2\hat F_1(-N,b;c;z)$ when
$N\in\mathbb{N}$ is defined for all $c\in\mathbb{C}$, and is
continuous in~$c$ as well as in~$b$.  This desingularization of the
parametric function~${}_2F_1$ differs from that of
Olver \cite[Chapter~15]{Olver2010}, in which it is multiplied by
$1/\Gamma(c)$ rather than $(c)^{\overline N}$.

\section{The Weyl--Heisenberg algebra}
\label{sec:algebra}

The (one-dimensional) Weyl--Heisenberg algebra $W_{\mathbb{C}}$ is a
unital associative algebra over the field of scalars~$\mathbb{C}$
generated by noncommuting indeterminates~$\adag,a$.  Any element
of~$W_{\mathbb{C}}$ is a combination over~$\mathbb{C}$ of finite
products of~$\adag,a$, and the unit (multiplicative identity) of the
algebra, conventionally written as~$1$, is the empty product
multiplied by the scalar~$1$.

The representation of elements of~$W_{\mathbb{C}}$ is not unique, due
to the existence of the commutation relation $[a,\adag]=1$.  One can
write
  \begin{equation}
    W_{\mathbb{C}} = \mathbb{C}\langle \adag,a\rangle \,/\, \mathbb{C}(a\adag - \adag
    a - 1),
  \end{equation}
where $\mathbb{C}\langle \adag,a\rangle$ is the free associative
unital algebra over~$\mathbb{C}$ generated by~$\adag,a$, and the ideal
quotiented~out is the two-sided one generated over~$\mathbb{C}$ by the
element $a\adag - \nobreak\adag a - \nobreak1$, where $1$~is the unit
of the free algebra.  By definition, an ordering identity
in~$W_{\mathbb{C}}$ is an element of this two-sided ideal.  For
example, one can prove by induction that for all $m,n\in\mathbb{N}$,
\begin{equation}
\label{eq:difflr}
  a^m \adag^n= \sum_{\ell=0}^{\min(m,n)} \ell\,!\, \binom{m}{\ell}\binom{n}{\ell}\,
  \adag^{(n-\ell)} a^{m-\ell}
\end{equation}
holds as an identity in~$W_\mathbb{C}$.  This expresses an
anti-normally ordered monomial in~$\adag,a$ in~terms of normally
ordered ones.  Viewed as an element of $\mathbb{C}\langle
\adag,a\rangle$, the difference between the left-hand and right-hand
sides of~(\ref{eq:difflr}) is an element of the ideal
$\mathbb{C}( a\adag -\nobreak \adag a -\nobreak 1)$.

The polynomial algebra over~$\mathbb{C}$ generated by $x,D=D_x={\rm
  d}/{\rm d}x$ with $[D,x]=1$ is isomorphic to~$W_{\mathbb{C}}$, but
including negative integer powers of~$x$ as elements extends it to an
algebra that has $x^{-1}$ as an additional generator, subject to the
relation $[D,x^{-1}]=-(x^{-1})^2$.  The extension algebra~$\tilde
W_{\mathbb{C}}$ used here is a larger, infinitely generated one: it
has each power $x^\alpha$, $\alpha\in\mathbb{C}$, as an element,
subject to $[D,x^\alpha]=\alpha(x^{\alpha-1})$.  It is also useful to
consider the deformation $W_{\mathbb{C}}(s)$ of~$W_{\mathbb{C}}$
defined by $UV-\nobreak VU=V^s$, where $s\in\mathbb{C}$.  As noted,
when $s\in\mathbb{N}$ it is isomorphic to a subalgebra
of~$W_{\mathbb{C}}$, with $V,U$ and $\adag,\adag^s a$ identified.
When $s\in\mathbb{C}\setminus{\mathbb{N}}$ it is a subalgebra
of~$\tilde W_{\mathbb{C}}$, with $V,U$ and $x,x^s D$ identified.
Owing to the embedding $W_{\mathbb{C}}(s)\hookrightarrow W_{\mathbb{C}}$
(resp.~$\tilde W_{\mathbb{C}}$), any ordering identity
in~$W_{\mathbb{C}}(s)$ can be viewed as an ordering identity
in~$W_{\mathbb{C}}$ (resp.~$\tilde W_{\mathbb{C}}$).

The algebras $W_{\mathbb{C}}$ and $\tilde W_{\mathbb{C}}$ are
respectively $\mathbb{Z}$\nobreakdash-graded and
$\mathbb{C}$\nobreakdash-graded.  For any $\omega\in{W}_{\mathbb{C}}$
that is a nonzero scalar multiple of a product of $\adag$'s and~$a$'s
(i.e., a monomial), let $|\omega|_\adag,\allowbreak |\omega|_a$ denote
the number of $\adag$'s and~$a$'s.  Similarly, if an element $\tilde
\omega\in\tilde{W}_{\mathbb{C}}$ is a nonzero multiple of a product of
powers of $x$ and~$D$, let $|\tilde \omega|_x,\allowbreak |\tilde
\omega|_D$ denote the sum of those powers: for instance,
$|x^{1/2}Dx^{-1/2}|_x=0$ and $|x^{1/2}Dx^{-1/2}|_D=1$.  The
\emph{excess} (or~degree) of~$\omega$ is
$e(\omega)\defeq|\omega|_\adag-\nobreak |\omega|_a\in\mathbb{Z}$, and
that of~$\tilde \omega$ is $e(\tilde \omega)\defeq|\tilde
\omega|_x-\nobreak |\tilde \omega|_D\in\mathbb{C}$.  These algebras
are graded by excess: any element can be written as a sum of groups of
terms, each group containing terms with a distinct excess.
For~$W_{\mathbb{C}}$, the quotienting~out of the ideal
$\mathbb{C}(a\adag -\nobreak \adag a -\nobreak 1)$
respects the grading, as each term in $a\adag -\nobreak \adag a
-\nobreak 1$ has zero excess; and similarly for~$\tilde
W_{\mathbb{C}}$.

As any ordering identity (i.e., element of the ideal) can be grouped
into terms with distinct excesses, it is reasonable to focus on ones
in which only a single excess is represented.  For instance, each term
in~(\ref{eq:difflr}) has excess $n-\nobreak m\in\mathbb{Z}$, and in
the identities (\ref{eq:katriel}), (\ref{eq:firstmain}),
(\ref{eq:firstmainmod}), (\ref{eq:sampleappl}) of the introduction,
the common excess is respectively $0,0,\allowbreak En,2n$.  The
algebra~$W_{\mathbb{C}}$ has a
subalgebra~$W_{\mathbb{C}}^{\textrm{bal}}$ (and $\tilde
W_{\mathbb{C}}$ has a subalgebra~$\tilde
W_{\mathbb{C}}^{\textrm{bal}}$) that comprises all linear combinations
of ``balanced'' or ``scale invariant'' monomials: ones with excess~$0$.
It can be shown that $W_{\mathbb{C}}^{\textrm{bal}}$ is singly
generated: it is $\mathbb{C}[\adag a]$, the algebra of
complex-coefficient polynomials in the boson number operator
$\mathcal{N}\defeq\adag a$, though it could also be written as
$\mathbb{C}[a \adag]$ or $\mathbb{C}[\adag a+ a\adag]$, etc.

Identities relevant to single-annihilator $\Omega(\adag,a)$ include
what is properly called the Tait--Toscano--Viskov
identity \cite{Mansour2016,Toscano55}, which is $(\adag a \adag)^n =
\adag^na^n\adag^n$ in~$W_{\mathbb{C}}$ and $(xDx)^n=x^nD^nx^n$
in~$\tilde W_{\mathbb{C}}$; it holds for all $n\in\mathbb{N}$.
Another pair of identities is
\begin{subequations}
  \begin{align}
    & \adag^L a \adag^R + \adag^R a \adag^L = 2(\adag)^{(L+R)/2} a (\adag)^{(L+R)/2},&&\text{in $W_{\mathbb{C}}$},\label{eq:otherpaira}\\
    &x^L D x^R + x^R D x^L = 2x^{(L+R)/2} D x^{(L+R)/2},&&\text{in $\tilde W_{\mathbb{C}}$}.\label{eq:otherpairb}
  \end{align}
\end{subequations}
These hold respectively for all $L,R\in\mathbb{N}$ with $L+\nobreak R$
even, and for all $L,R\in\mathbb{C}$.  The following will play a major
role in deriving operator ordering identities from polynomial
identities.

\begin{lemma}
\label{lem:major}
  In\/ $\tilde W_{\mathbb{C}}^{\rm{bal}}$, the identities
  \begin{subequations}
  \begin{align}
    \label{eq:majorrole1a}
    (xD)^{\underline n,\alpha} &= (x^\alpha)^n (x^{1-\alpha}D)^n,\\
    &= (xDx^\alpha)^n (x^{-\alpha})^n
    \label{eq:majorrole1b}
  \end{align}
  \end{subequations}
  hold for all $n\in\mathbb{N}$ and $\alpha\in\mathbb{C}$.  More generally,
  \begin{subequations}
  \begin{align}
    \label{eq:majorrole2a}
    (xD+r)^{\underline n,\alpha} &= x^{-r}\left[(x^\alpha)^n (x^{1-\alpha}D)^n\right]x^r
    =(x^\alpha)^n(x^{1-\alpha-r}Dx^r)^n,
    \\
    &= x^{-r}\left[ (xDx^\alpha)^n (x^{-\alpha})^n \right] x^r
    = (x^{1-r}Dx^{\alpha+r})^n (x^{-\alpha})^n
    \label{eq:majorrole2b}
    \end{align}
    \end{subequations}
  hold for all $n\in\mathbb{N}$ and $\alpha,r\in\mathbb{C}$, with
  \begin{equation}
    \label{eq:powerfullemma}
    (x^{1-L-R})^n(x^LDx^R)^n = (x^{1-R}Dx^{1-L})^n (x^{L+R-1})^n
  \end{equation}
  being a reparametrized statement of the equality between\/
  {\rm(\ref{eq:majorrole2a})} and\/~{\rm(\ref{eq:majorrole2b})}, which
  holds for all $n\in\mathbb{N}$ and $L,R\in\mathbb{C}$.
\end{lemma}

When $\alpha=1$, (\ref{eq:majorrole1a})~specializes to $(xD)^{\underline
  n}=x^nD^n$, which is the well-known identity of
Boole~\cite{Mansour2016}: in~the $W_{\mathbb{C}}^{\text{bal}}$ form
$(\adag a)^{\underline n}=\adag^n a^n$, it normally-orders any falling
factorial power of the number operator.  But
(\ref{eq:majorrole1a}),(\ref{eq:majorrole1b}) are not so well known,
though (\ref{eq:majorrole1a})~appears as Proposition~5
of~\cite{MohammadNoori2011}.  The versions
(\ref{eq:majorrole2a}),(\ref{eq:majorrole2b}) come from them by
similarity transformations.  All these identities can be proved by
induction, or (in~$\tilde W_{\mathbb{C}}$) operationally: showing that
the left and right-hand sides act identically on monomials or
polynomials in~$x$~\cite{Carlitz76}.

The algebras $W_{\mathbb{C}}$ and $\tilde W_{\mathbb{C}}$ have
nontrivial automorphisms, but the only one that will appear here is
the formal adjoint operation~$*$ on~$\tilde W_{\mathbb{C}}$ (strictly,
an anti-automorphism), which treats its elements as differential
operators.  For instance, $\left[ (x^L D x^R)^n\right]^*$ equals
$(-1)^{n} (x^R D x^L)^n$.  Up~to parameter renamings, the adjoint
operation~$*$ transforms (\ref{eq:majorrole2a})
to~(\ref{eq:majorrole2b}), and vice versa, and
(\ref{eq:powerfullemma})~is self-adjoint: it~is transformed to itself.

\section{Generalized Stirling and Eulerian numbers}
\label{sec:fundamentals}

The unified generalized Stirling numbers
$S_{n,k}=S_{n,k}(\alpha,\beta;r)$ of Hsu and Shiue~\cite{Hsu98}, and
related generalized Eulerian numbers
$E_{n,k}=E_{n,k}(\alpha,\beta;r)$, both of which are indexed by $n,k$
with $0\le k\le n$, can be defined in several ways.  They are given by
finite sums (see Theorem~\ref{thm:altdef1} below), and they satisfy
triangular recurrences (see Theorem~\ref{thm:recurrences}), which can
be used to compute them when $n$ and~$k$ are small.  But they can also
be defined abstractly, as coefficients that connect alternative bases
of the linear space of polynomials of degree~$\le n$ in an
indeterminate~$x$.  This leads to many identities in polynomial
algebra.

These numbers will be viewed as elements of infinite, lower triangular
matrices $\mathcal{S}=[S_{n,k}]$ and $\mathcal{E}=[E_{n,k}]$, indexed
by $n,k\in\mathbb{N}$.  The triangles $\mathcal{S}$ and~$\mathcal{E}$
are parametrized by $\alpha,\beta,r\in\mathbb{C}$, which specify the
bases.  As will be seen, the modified triangle
$\hat{\mathcal{S}}=[\hat S_{n,k}]\defeq[\beta^kk!\,S_{n,k}]$ and the
triangle~$\mathcal{E}$ are in a sense dual to each other.  To~reveal
the symmetry, many identities in this section are written in~terms of
$\hat S_{n,k}$ rather than~$S_{n,k}$.  Though a (weak) case can be
made that the generalized Stirling numbers should be redefined to
include the factor~$\beta^kk!$, the original
normalization~\cite{Hsu98} will be maintained here.

In the linear-algebraic approach, the numbers $\hat S_{n,k}$
and~$E_{n,k}$, $0\le k\le n$, are defined to satisfy
\begin{subequations}
\begin{align}
\label{eq:Sdef2}
(\beta x+r)^{\underline n,\alpha} &= \sum_{k=0}^n \hat
S_{n,k}(\alpha,\beta;\,r)\binom{x}{k},\\
\label{eq:Edef2}
(\beta x+r)^{\underline n,\alpha} &= \sum_{k=0}^n
E_{n,k}(\alpha,\beta;\,r)\,\binom{x+n-k}{n}.
\end{align}
\end{subequations}
The sets of polynomials $\left[\binom{x}k\right]_{k=0}^n$ and
$\left[\binom{x+n-k}n\right]_{k=0}^n$ are bases of the space of
polynomials of degree~$\le n$ in~$x$ (the latter statement requiring a
proof~\cite{Carlitz78}), so $\hat S_{n,k}$ and~$E_{n,k}$ are uniquely
defined.  They are continuous functions of $\alpha,\beta,r$.  If
$\beta\neq0$, replacing $x$ by $(x-\nobreak r)/\beta$ and defining
$S_{n,k}\defeq \hat S_{n,k}/\beta^kk!$ produces the less symmetrical pair
of definitions
\begin{subequations}
\begin{align}
\label{eq:Sdef1}
(x)^{\underline n,\alpha} &= \sum_{k=0}^n S_{n,k}(\alpha,\beta;\,r)\,(x-r)^{\underline k,\beta},
\\
\label{eq:Edef1}
\beta^nn!\,(x)^{\underline n,\alpha} &= \sum_{k=0}^n E_{n,k}(\alpha,\beta;\,r)\,\left[x-r+\beta(n-k)\right]^{\underline n,\beta}.
\end{align}
\end{subequations}
The definition~(\ref{eq:Sdef1}) is that of Hsu and Shiue, and it
defines the $S_{n,k}$ uniquely even when $\beta=0$, as continuous
functions of $\alpha,\beta,r$.  (Note that if $\beta=0$, $\hat
S_{n,k}$ vanishes unless $k=0$, but the same is not true
of~$S_{n,k}$.)  A~consequence of~(\ref{eq:Sdef1}) is that the
parametric triangle $\mathcal{S}=[S_{n,k}]$ satisfies
\begin{equation}
\label{eq:Smultrule}
  \mathcal{S}(\alpha,\gamma;\,r_1+r_2)= \mathcal{S}(\alpha,\beta;\,r_1)\,\mathcal{S}(\beta,\gamma;\,r_2),
\end{equation}
matrix multiplication being meant.  In accordance
with~(\ref{eq:Smultrule}), $\mathcal{S}(\alpha,\alpha;0)$ for
any~$\alpha$ equals the infinite identity matrix~$\mathcal{I}$, and
for any~$\alpha,\beta,r$, the inverse matrix
$\mathcal{S}(\alpha,\beta;r)^{-1}$ exists and equals
$\mathcal{S}(\beta,\alpha;-r)$.  

For the~$E_{n,k}$, the definition~(\ref{eq:Edef1}) has been used by
the author~\cite{Maier28}.  It defines the $E_{n,k}$ uniquely when
$\beta\neq0$, though (\ref{eq:Edef2}) must be used when~$\beta=0$.
When $(\alpha,\beta;r)=(0,1;0)$, it reduces to Worpitzky's
identity~\cite{Graham94}, which is satisfied by the classical Eulerian
numbers.  The case when $\alpha=0$ but $\beta,r$ are arbitrary, which
is intermediate between the classical case and the present one, has
recently been treated~\cite{Mansour2019,Mezo2016,Ramirez2018}.

Eq.~(\ref{eq:Sdef2}) has a matrix-algebraic interpretation.  Let
$\mathcal{V}(\alpha,\beta;r)=\allowbreak\left[(\beta x+\nobreak r)^{\underline
    n,\alpha}\right]$ be an infinite matrix of generalized Vandermonde type
indexed by $n,x\in\mathbb{N}$, and let
$\mathcal{P}=\left[{\binom{x}{k}}\right]$ be an infinite Pascal matrix
indexed by $x,k\in\mathbb{N}$, which is lower triangular (nonzero
elements having $0\le k\le x$).  Also, let ${\mathcal{P}}^t$ be its
upper-triangular transpose.  Then (\ref{eq:Sdef2})~says that
\begin{equation}
\label{eq:LU}
  \mathcal{V}(\alpha,\beta;\,r) = \hat{\mathcal{S}}(\alpha,\beta;\,r)\,{\mathcal{P}}^t,
\end{equation}
which defines $\hat{\mathcal{S}}(\alpha,\beta;r)$ as the
$\text{L}$\nobreakdash-factor in a certain $\text{LU}$~decomposition
of $\mathcal{V}(\alpha,\beta;r)$.  For a related interpretation
of~(\ref{eq:Edef2}), see the remark after Theorem~\ref{thm:1} below.

It follows by inspection that irrespective of the choice of parameters
$\alpha,\beta,r$, the apex elements $S_{0,0}$, $\hat S_{0,0}$, and
$E_{0,0}$ of the triangles $\mathcal{S}$,
$\hat{\mathcal{S}}$,~$\mathcal{E}$ equal unity.  The homogeneity
properties
\begin{equation}
\begin{gathered}
S_{n,k}(\lambda \alpha,\lambda \beta;\,\lambda r) = \lambda^{n-k} S_{n,k}(\alpha,\beta;\,r)  
\\
\hat S_{n,k}(\lambda \alpha,\lambda \beta;\,\lambda r) = \lambda^n \hat S_{n,k}(\alpha,\beta;\,r),\qquad
E_{n,k}(\lambda \alpha,\lambda \beta;\,\lambda r) = \lambda^n E_{n,k}(\alpha,\beta;\,r)  
\label{eq:homogeneity}
\end{gathered}
\end{equation}
also follow by inspection.  Each $S_{n,k}$ and $E_{n,k}$ turns~out to
be a polynomial in~$\alpha,\beta,r$ with integral coefficients (see
below); so if $\alpha,\beta,r$ are integers, each $S_{n,k}$ and
$E_{n,k}$ will be an integer.

For some choices of $\alpha,\beta,r\in\mathbb{Z}$, all $S_{n,k}$
and/or $E_{n,k}$ may be non-negative integers, in which case they may
have combinatorial interpretations.  This is the case for many
$\alpha,\beta,r$, as can be confirmed by comparing the recurrences
satisfied by the $S_{n,k}$ and $E_{n,k}$ (see
Theorem~\ref{thm:recurrences} below) with ones that have appeared in
the literature \cite{Comtet74,OEIS2023}.  Examples of such
triangularly indexed combinatorial numbers include the Stirling subset
numbers $\stirsub{n}k=S_{n,k}(0,1;0)$, which count the partitions of
an $n$\nobreakdash-set into $k$~non-empty blocks.  They are often
called the Stirling numbers of the second kind and are also denoted by
$S(n,k)$.  (When $(\alpha,\beta;r)=(0,1;0)$, the Vandermonde
factorization~(\ref{eq:LU}) is well known~\cite{Cheon2001}.) Another
example is provided by the Eulerian numbers $\eulerian{n}k =
E_{n,k}(0,1;1)$, which count the permutations of an
$n$\nobreakdash-set that have exactly $k$~descents
\cite{Graham94,Riordan58}. The traditionally indexed Eulerian numbers
$A_{n,k}$ are horizontally shifted versions: $A_{n,k+1}$
equals~$\eulerian{n}k$, with $A_{n,0}=0$ unless~$n=0$; and
$A_{n,k}=E_{n,k}(0,1;0)$.

Known $S_{n,k}$ with $\beta=0$ include $\binom{n}k=S_{n,k}(0,0;1)$ and
the Stirling cycle numbers $\stircyc{n}k=S_{n,k}(-1,0;0)$, which count
the permutations of an $n$\nobreakdash-set, the cycle decompositions
of which have exactly $k$~cycles.  They are often called the unsigned
Stirling numbers of the first kind and are also denoted by
$(-1)^{n-k}s(n,k)$, the numbers $s(n,k)=S_{n,k}(1,0;0)$ being the
signed versions.  It follows from~(\ref{eq:Smultrule}) that
$\mathcal{S}(\alpha,\beta;\,r) = \mathcal{S}(\alpha,0;\,0)\,
\mathcal{S}(0,0;\,r)\, \mathcal{S}(0,\beta;\,0)$, i.e.,
\begin{equation}
\label{eq:elegant}
  S_{n,k}(\alpha,\beta;\,r) = \sum_{j=k}^n\sum_{p=k}^j
  (-\alpha)^{n-j}\stircyc{n}{j} r^{j-p}\binom{j}{p}
  \beta^{p-k}\stirsub{p}k.
\end{equation}
The existence of this elegant decomposition
formula \cite{Spivey2011,Pan2012} justifies keeping the Hsu--Shiue
definition of~$S_{n,k}$, as $\stircyc{n}j$~can be expressed in~terms
of~$\hat S_{n,j}$ only by taking a $\beta\to0$ limit.

Relatively simple formulas for the $\hat S_{n,k}$ and $E_{n,k}$ can be
obtained by extracting coefficients from (\ref{eq:Sdef2})
and~(\ref{eq:Edef2}).  The following can be used to compute $\hat
S_{n,k}$ (and $S_{n,k}$ if~$\beta\neq0$), and $E_{n,k}$, for all $k$ with
$0\le k\le n$, though the effort needed grows rapidly with $n$
and~$k$.
\begin{theorem}
\label{thm:altdef1}
  The triangles\/ $\hat{\mathcal{S}}=[\hat S_{n,k}]$ and\/
  $\mathcal{E}=[{E}_{n,k}]$ are alternatively defined by
\begin{subequations}
\begin{align}
&
\label{eq:Sform}
\begin{aligned}
\hat S_{n,k}(\alpha,\beta;\,r) &= \Delta_x^k\left[(\beta x+r)^{\underline n,\alpha}\right]\bigm|_{x=0}
\\
&= \sum_{x=0}^k (-1)^{k-x} \binom{k}{x} (\beta x+r)^{\underline n,\alpha},
\end{aligned}
\\
\shortintertext{and}
&
\label{eq:Eform}
\begin{aligned}
E_{n,k}(\alpha,\beta;\,r) &= \nabla_x^{n+1}\left[(\beta x+r)^{\underline n,\alpha}\right]\bigm|_{x=k}
\\
&= \sum_{x=0}^k (-1)^{k-x} \binom{n+1}{k-x} (\beta x+r)^{\underline n,\alpha},
\end{aligned}
\end{align}
\end{subequations}
with each right-hand side equaling zero unless\/ $0\le k\le n$.
\end{theorem}
\begin{proof}
Acting on functions of $x=0,1,2,\dots$, the forward and backward
difference operators, $\Delta_x$~and~$\nabla_x$, are defined by
$(\Delta_x f)(x) =\allowbreak f(x+\nobreak 1)-\nobreak f(x)$ and
$(\nabla_x f)(x) =\allowbreak f(x) -\nobreak f(x-\nobreak 1)$, with
$f(-1)$ taken to equal zero in the latter.  From
$\Delta_x^j\left[(x)^{\underline k}\right] = k^{\underline
  j}(x)^{\underline{k-j}}$, it follows readily that
\begin{subequations}
\begin{align}
  &\frac1{j!}\Delta_x^j\left[(x)^{\underline k}\right]\bigm|_{x=0} = \delta_{j,k},\label{eq:extract1}\\
  &\frac1{n!}\nabla_x^{n+1}\Bigl\{\left[ x+(n-k)\right]^{\underline n}\Bigr\}\bigm|_{x=j} = \delta_{j,k},
\label{eq:extract2}
\end{align}
\end{subequations}
where $\delta_{j,k}$ is the Kronecker delta.  These statements hold
for all $j,k,n\in\mathbb{N}$, with $k\le n$ assumed in the latter.
Applying (\ref{eq:extract1}),(\ref{eq:extract2}) to
(\ref{eq:Sdef2}),(\ref{eq:Edef2}) yields
(\ref{eq:Sform}),(\ref{eq:Eform}).
\end{proof}

Formulas (\ref{eq:Sform}) and~(\ref{eq:Eform}) are generalizations of
known ones for the Stirling subset and classical Eulerian
numbers~\cite{Graham94}. For each value of the row
index~$n\in\mathbb{N}$, they express $\hat S_{n,k} =
\beta^kk!\,S_{n,k}$ and~$E_{n,k}$, viewed as functions of the column
index $k=0,\dots,n$, as binomial transforms of $(\beta x+\nobreak
r)^{\underline n,\alpha}$, viewed as a function of~$x=0,\dots,n$.  In
particular, (\ref{eq:Sform})~says that
$\hat{\mathcal{S}}=\mathcal{V}\left({\mathcal{P}^{-1}}\right)^t$.
Having the factor $(-1)^{k-x}$ in their summands,
(\ref{eq:Sform})~and~(\ref{eq:Eform}) are properly called
\emph{inverse} binomial transforms.  According to the general theory
of binomial transforms~\cite{Boyadzhiev2018}, they are inverse to the
forward transforms (\ref{eq:Sdef2}) and~(\ref{eq:Edef2}), which
include no~such factor, and are their immediate corollaries.

In fact, corresponding rows of the triangles $\hat{\mathcal{S}}=[\hat
  S_{n,k}]$ and $\mathcal{E}=[E_{n,k}]$ are binomial transforms of
each other, so each triangle can be computed from the other.  The
following theorem, the $\alpha=0$ case of which has previously
appeared \cite{Mezo2016,Ramirez2018}, generalizes the known pair of
relations~\cite{Graham94} between the Stirling subset numbers
$\stirsub{n}k=S(n,k)=S_{n,k}(0,1;0)$ and the traditionally indexed
Eulerian numbers $A_{n,k}=\eulerian{n}{k-1}=E_{n,k}(0,1,0)$.  The
involutive ``rev'' transformation
$\mathcal{M}\mapsto{\mathcal{M}}^{\rm rev}$ is defined to produce from
any infinite lower-triangular matrix a like matrix in which each row
is reversed.  If $\mathcal{M}=[M_{n,k}]$, then ${\mathcal{M}}^{\rm
  rev}=[M^{\rm rev}_{n,k}]$ has $M^{\rm rev}_{n,k} = M_{n,n-k}$,
$k=0,\dots,n$.

\begin{theorem}
\label{thm:1}
For all\/ $n\ge0$ and\/ $\alpha,\beta,r$, one has
\begin{subequations}
  \begin{align}
    \hat S_{n,k} &= \sum_{j=0}^k \binom{n-j}{n-k} E_{n,j},\label{eq:EtoS}\\
    E_{n,k} &= \sum_{j=0}^k (-1)^{k-j} \binom{n-j}{n-k}\, \hat S_{n,j},\label{eq:StoE}
  \end{align}
\end{subequations}
which are a binomial transform and its inverse.  That is,
$\hat{\mathcal{S}}^{\rm{rev}} = {\mathcal{E}}^{\rm{rev}}\,\mathcal{P}$
and ${\mathcal{E}}^{\rm{rev}} =
\hat{\mathcal{S}}^{\rm{rev}}\,\mathcal{P}^{-1}$.
\end{theorem}
\begin{proof}
The basis functions (of~$x$) that appear in the original definitions
(\ref{eq:Sdef2}) and~(\ref{eq:Edef2}) are related by the identities
\begin{subequations}
\begin{align}
  &\binom{x+n-k}{n} = \sum_{j=k}^n \binom{n-k}{n-j}\binom{x}{j},\label{eq:elident1}\\
  &\binom{x}{k} = \sum_{j=k}^n (-1)^{j-k}\binom{n-k}{n-j}\binom{x+n-j}{n}.\label{eq:elident2}
\end{align}
\end{subequations}
The first is a version of Vandermonde's identity, which has an
elementary combinatorial proof; one can also view it as a binomial
transform, in which case the second is its inverse.  (It~will be
recalled that if $\mathcal{P} = \left[{\binom{x}k}\right]$, indexed
by~$x,k$, then $\mathcal{P}^{-1} =
\left[(-1)^{k-x}{\binom{k}x}\right]$, indexed by~$k,x$.)  To verify
(\ref{eq:EtoS}), substitute it into~(\ref{eq:Sdef2}) and use
(\ref{eq:elident1}) to reduce the double sum to a single sum; one
obtains~(\ref{eq:Edef2}).  To verify~(\ref{eq:StoE}), substitute it
into~(\ref{eq:Edef2}) and use (\ref{eq:elident2}) similarly; one
obtains~(\ref{eq:Sdef2}).
\end{proof}
\begin{remarkaftertheorem}
As $\hat{\mathcal{S}}^{\text{rev}} =
{\mathcal{E}}^{\text{rev}}\,\mathcal{P}$, the Vandermonde
decomposition~(\ref{eq:LU}) can be written as
\begin{equation}
  \mathcal{V}(a,b;\,r) = \left({\mathcal E}^\text{rev}\,\mathcal{P}\right)^{\text{rev}}
\mathcal{P}^t.
\end{equation}
This is the promised matrix-algebraic restatement of~(\ref{eq:Edef2}),
which defines the matrix $\mathcal{E}=\mathcal{E}(a,b;r)$.
Equivalently, ${\mathcal{E}}^{\rm rev} = \left[\mathcal{V}
  \left(\mathcal{P}^{-1}\right)^t\right]^{\rm rev}\mathcal{P}^{-1}$,
which is a restatement of the summation formula~(\ref{eq:Eform}).
\end{remarkaftertheorem}

\smallskip
One can work~out generating functions for the numbers $\hat S_{n,k}$
and~$E_{n,k}$, and triangular recurrences which they satisfy.  Define
the parametric row polynomials $\hat S_n(t)$ and~$E_n(t)$,
$n\in\mathbb{N}$, by
  \begin{equation}
    \hat S_n(t) = \sum_{k=0}^n \hat S_{n,k}\,t^k,\qquad
    E_n(t) = \sum_{k=0}^n E_{n,k}\,t^k,
    \end{equation}
and the parametric bivariate exponential generating functions (EGF's)
$\hat S(t,z)$ and $E(t,z)$ by
  \begin{equation}
    \hat S(t,z) = \sum_{n=0}^\infty\sum_{k=0}^n \hat S_{n,k}\,t^k\frac{z^n}{n!},\qquad
    E(t,z) = \sum_{n=0}^\infty\sum_{k=0}^n E_{n,k}\,t^k\frac{z^n}{n!}.
  \end{equation}
(Each depends on $\alpha,\beta,r$.)  It is not difficult to see that the pair
\begin{equation}
  \hat S_n(t) = (1+t)^nE_n\left(t/(1+t)\right),\qquad
  E_n(t) = (1-t)^n\hat S_n\left(t/(1-t)\right)
\end{equation}
and the pair
\begin{equation}
\label{eq:relationpair}
  \hat S(t,z) = E\left( t/(1+t), (1+t)z\right),\qquad
   E(t,z) = \hat S\left( t/(1-t), (1-t)z\right)
\end{equation}
are each equivalent to (\ref{eq:EtoS}),(\ref{eq:StoE}).

\begin{theorem}
\label{thm:fromvertical}
When\/ $\alpha\neq0$ and\/ $\beta\neq0$, one has the explicit EGF's
\begin{subequations}
\label{eq:SandEGF}
  \begin{align}
    \hat S(t,z) &=\frac{(1+\alpha z)^{r/\alpha}}{1-t\left[(1+\alpha z)^{\beta/\alpha}-1\right]}, \label{eq:SEGF} \\
    E(t,z) &= \frac{(1-t)(1+\alpha z-\alpha tz)^{r/\alpha}}{1-t(1+\alpha z-\alpha tz)^{\beta/\alpha}}.
    \label{eq:EEGF}
  \end{align}
\end{subequations}
{\rm(}Taking the $\alpha\to0$ limit is straightforward; the $\beta\to0$ limit is not
discussed here.{\rm)}
\end{theorem}
\begin{proof}
Define the forward difference operator~$\Delta_{r,\beta}$, acting on
functions of~$r$, by $(\Delta_{r,\beta} f)(r) =\allowbreak f(r+\nobreak
\beta)-\nobreak f(r)$.  Then by~(\ref{eq:Sform}),
\begin{equation}
\hat S_{n,k}(\alpha,\beta;\,r) = \Delta_x^k\left[(\beta x+r)^{\underline n,\alpha}\right]\bigm|_{x=0}
= \Delta_{r,\beta}^k\left[(r)^{\underline n,\alpha}\right]
\end{equation}
which vanishes if $k>n$, so one has the ``vertical'' univariate EGF
\begin{equation}
\begin{aligned}
  &
  \sum_{n=k}^\infty \hat S_{n,k}(\alpha,\beta;\,r)\frac{z^n}{n!}
  =\Delta^k_{r,\beta}\left[
  \sum_{n=0}^\infty 
    (r)^{\underline n,\alpha}
  \frac{z^n}{n!}
  \right]
  =
  \Delta^k_{r,\beta}\left[
    (1+\alpha z)^{r/\alpha}
    \right]
  \\
  &
  \qquad=
  \sum_{j=0}^k(-1)^{k-j}\binom{k}j (1+\alpha z)^{(r+j\beta)/\alpha}
  =(1+\alpha z)^{r/\alpha}\left[(1+\alpha z)^{\beta/\alpha}-1\right]^k.
  \end{aligned}
\end{equation}
Multiplying this by $t^k$ and summing from $k=0$ to~$\infty$
yields the bivariate EGF~(\ref{eq:SEGF}), and (\ref{eq:EEGF}) comes by
applying~(\ref{eq:relationpair}).
\end{proof}

By direct calculation, $\hat S(t,z)$ and~$E(t,z)$ given by
(\ref{eq:SEGF}) and~(\ref{eq:EEGF}) satisfy the partial differential
equations (PDE's)
\begin{subequations}
\label{eq:PDEs}
  \begin{align}
    \frac{\partial\hat S}{\partial z}
    &= -\alpha z \frac{\partial\hat S}{\partial z} + \beta(1+t)t\frac{\partial \hat S}{\partial t} +(r+\beta t)\hat S,\\
    \frac{\partial E}{\partial z} &=
    \left[-\alpha+(\alpha+\beta)t\right]z\frac{\partial E}{\partial z} + \beta(1-t)t\frac{\partial E}{\partial t} + \left[r+(\beta-r)t\right]E.
  \end{align}
\end{subequations}

\begin{theorem}
\label{thm:recurrences}
The numbers\/ $\hat S_{n,k}$ and\/ $E_{n,k}$ satisfy 
\begin{subequations}
  \begin{align}
    \hat S_{n+1,k+1} &= \left[-\alpha n+\beta (k+1)+r\right] \hat S_{n,k+1} + \beta(k+1)\hat S_{n,k},\label{eq:SrecGKP}\\
    E_{n+1,k+1} &= \left[-\alpha n+\beta(k+1)+r\right] E_{n,k+1} + \left[(\alpha+\beta)n-\beta k+(\beta-r)\right] E_{n,k},\label{eq:ErecGKP}
  \end{align}
\end{subequations}
\vskip-\jot
\noindent
and the unmodified numbers $S_{n,k}$ satisfy
\begin{equation}
\label{eq:SGKPrec}
S_{n+1,k+1} = \left[-\alpha n+\beta (k+1)+r\right]  S_{n,k+1} +  S_{n,k},
\end{equation}
all three parametric recurrences being of the triangular, GKP type.
\end{theorem}
\begin{proof}
As $\hat S_{n,k}$, $E_{n,k}$ are the coefficients of $t^kz^n/n!$ in
$\hat S(t,z)$, $E(t,z)$, (\ref{eq:SrecGKP}) and~(\ref{eq:ErecGKP})
follow from~(\ref{eq:PDEs}) by equating coefficients.  If $\beta\neq0$,
(\ref{eq:SGKPrec})~follows from~(\ref{eq:SrecGKP}) by $\hat S_{n,k} =
\beta^k k!\, S_{n,k}$; and if $\beta=0$, by continuity.
\end{proof}

The recurrences (\ref{eq:SrecGKP}) and~(\ref{eq:ErecGKP}) can be used
as alternative definitions of the $\hat S_{n,k}$ and~$E_{n,k}$.  By
starting with the apex elements ${\hat S_{0,0}=1}$ and ${E_{0,0}=1}$,
and iterating, one can compute any desired $\hat S_{n,k}$
and~$E_{n,k}$, though the effort needed grows rapidly with $n$
and~$k$.  Each $\hat S_{n,k}$ and~$E_{n,k}$ will be a homogeneous
polynomial in~$\alpha,\beta,r$, of joint degree~$n$ and with integral
coefficients.  The useful reversal (or ``reflection'') identity
$\mathcal{E}^{\text{rev}}(\alpha,\beta;r) = \mathcal{E}(-\alpha,\beta;\beta-\nobreak r)$,
i.e.,
\begin{equation}
\label{eq:reflectionprop}
  E_{n,n-k}(\alpha,\beta;\,r) = E_{n,k}(-\alpha,\beta;\,\beta-r),
\end{equation}
follows from (\ref{eq:ErecGKP}) by substituting $n-\nobreak k$
for~$k$.  When $(\alpha,\beta;r)=(0,1;1)$, it reduces to the known reversal
identity $\eulerian{n}{n-k}=A_{n,k}$, i.e., the property
$\eulerian{n}{n-k}=\eulerian{n}{k-1}$ of the classical Eulerian
numbers~\cite{Graham94}.

$S_{n,k}$ is similarly a homogeneous polynomial in~$\alpha,\beta,r$
with integral coefficients, but it is of joint degree~$n-\nobreak k$,
as follows from the decomposition formula~(\ref{eq:elegant}) and also
from~(\ref{eq:SGKPrec}).  By applying (\ref{eq:SGKPrec}) and
(\ref{eq:SrecGKP}),(\ref{eq:ErecGKP}) repeatedly, one finds that the
elements on the left edges of these triangles are $S_{n,0} = \hat
S_{n,0} = E_{n,0} = (r)^{\underline n,\alpha}$, and those on the right
edges are $S_{n,n}=1$, $\hat S_{n,n}=\beta^nn!$, and
$E_{n,n}=(\beta-\nobreak r)^{\overline n,\alpha}$.

It is worth noting here that the
generalized Vandermonde factorization~(\ref{eq:LU}) can be written
in~terms of the unmodified matrix~$\mathcal{S}$ as
\begin{equation}
  \label{eq:LDU}
  \mathcal{V}(\alpha,\beta;\,r) = \mathcal{S}(\alpha,\beta;\,r)
  \mathcal{D}(\beta){\mathcal{P}}^t,
\end{equation}
where $\mathcal{D}(\beta)$ is the infinite diagonal matrix with
$(k,k)$ element equal to~$\beta^k/k!$.  This has the form of a
standard $\text{LDU}$ decomposition, because each diagonal element of
the $\text{L}$\nobreakdash-factor~$\mathcal{S}$ equals unity, as is
the case with the $\text{U}$\nobreakdash-factor~$\mathcal{P}^t$.  When
$(\alpha,\beta;r)$ equals $(0,1;0)$ and
$\mathcal{V}(\alpha,\beta;r)=\allowbreak \left[(\beta x+\nobreak
  r)^{\underline n,\alpha}\right]$ becomes the infinite Vandermonde
matrix $\left[x^n\right]$, indexed by $n,x\in\mathbb{N}$, it
specializes to the well-known Vandermonde factorization mentioned
above, in which $\mathcal{S}$~is the matrix of Stirling numbers of the
second kind~\cite{Cheon2001}.

\smallskip
Additional identities can be deduced from the bivariate EGF's
in~(\ref{eq:SandEGF}).  Again by direct calculation, they satisfy
difference equations on~$r$, viz.,
\begin{subequations}
\begin{gather}
\label{eq:Sdiffrec}
\begin{aligned}
  &\Delta_{r,\alpha}\hat S(t,z) = \alpha z\hat S(t,z),\qquad  t\Delta_{r,\beta}\hat S(t,z) = \hat S(t,z)-\hat S(0,z),
\end{aligned}
\\
\shortintertext{and}
\label{eq:Ediffrec}
\begin{aligned}
  \begin{aligned}
    &\Delta_{r,\alpha} E(t,z) = \alpha(1-t)z\,E(t,z),\\
    &\begin{aligned}
       t\Delta_{r,\beta}E(t,z) &= (1-t)\left[E(t,z) -E\left(0, (1-t)z\right)\right]
       \\
       &= (1-t)\left[E(t,z)-(1+\alpha z-\alpha tz)^{r/\alpha}\right],
     \end{aligned}
  \end{aligned}
\end{aligned}
\end{gather}
\end{subequations}
with (\ref{eq:Sdiffrec}) and~(\ref{eq:Ediffrec}) related to each other
by~(\ref{eq:relationpair}).  Equating coefficients yields the
following.

\begin{theorem}
The parametric triangles\/ $\hat{\mathcal S}=[\hat S_{n,k}]$ and\/
${\mathcal E}=[E_{n,k}]$ satisfy the difference equations
\begin{subequations}
\label{eqs:Sdiffonr}
\begin{gather}
  \begin{aligned}
    & \Delta_{r,\alpha} \hat S_{n+1,k}(\alpha,\beta;\,r) = \alpha(n+1)\hat S_{n,k}(\alpha,\beta;\,r),
    \label{eq:Sdiffonr}
    \\
    & \Delta_{r,\beta} \hat S_{n,k}(\alpha,\beta;\,r) = \hat S_{n,k+1}(\alpha,\beta;\,r),
  \end{aligned}
\\
\shortintertext{and}
  \begin{aligned}
    & \Delta_{r,\alpha} E_{n+1,k+1}(\alpha,\beta;\, r) = \alpha(n+1)\left[E_{n,k+1}-E_{n,k}\right](\alpha,\beta;\,r),
    \\
    & 
    \begin{aligned}
      \Delta_{r,\beta} E_{n,k}(\alpha,\beta;\,r) &= \left[E_{n,k+1}-E_{n,k}\right](\alpha,\beta;\,r) +(-1)^k \binom{n+1}{k+1} E_{n,0}(\alpha,\beta;\,r)   \\
      & = \left[E_{n,k+1}-E_{n,k}\right](\alpha,\beta;\,r) +(-1)^k \binom{n+1}{k+1}(r)^{\underline n,\alpha}
    \end{aligned}
  \end{aligned}
\label{eq:gaa}
\end{gather}
\end{subequations}
on the third parameter~$r$.
\end{theorem}
These difference equations can be iterated, so as to increment or
decrement the parameter~$r$ of any $\hat{S}_{n,k}(\alpha,\beta;r)$ or
$E_{n,k}(\alpha,\beta;r)$ by any desired positive integer multiple of
$\alpha$~or~$\beta$.

One can also construct Newton series, which are finite-difference
analogues of Maclaurin or Taylor series.  For a function $f(r)$, the
Newton series at $r=0$ is
\begin{equation}
\label{eq:diffMacl}
  f(r) = \sum_{m=0}^\infty \frac{\Delta_{r,c}^mf(0)}{c^mm!} (r)^{\underline m,c},
\end{equation}
where $c\neq0$ is arbitrary.  It follows from (\ref{eq:Sdiffonr})
that
\begin{equation}
  \begin{aligned}
    & \Delta_{r,\alpha}^m \hat S_{n,k}(\alpha,\beta;\,0) = \alpha^m(n)^{\underline m}\,\hat S_{n-m,k}(\alpha,\beta;\,0),
    \\
    & \Delta_{r,\beta}^m \hat S_{n,k}(\alpha,\beta;\,0) = \hat S_{n,k+m}(\alpha,\beta;\,0),
  \end{aligned}
\end{equation}
and (\ref{eq:diffMacl}) therefore yields the formulas
\begin{equation}
  \begin{aligned}
    & 
    \hat{S}_{n,k}(\alpha,\beta;\,r) = \sum_{m=0}^{n-k} \binom{n}{m} \hat S_{n-m,k}(\alpha,\beta;\,0)\,(r)^{\underline m,\alpha},
    \\
    &
    \hat{S}_{n,k}(\alpha,\beta;\,r) = \sum_{m=0}^{n-k} \frac{1}{\beta^mm!} \hat S_{n,k+m}(\alpha,\beta;\,0)\,(r)^{\underline m,\beta},
  \end{aligned}
\label{eq:Shatsumformulas}
\end{equation}
it being assumed in the latter than $\beta\neq0$.  Thus from any row
or column of the matrix $\hat{\mathcal{S}}(\alpha,\beta;0)$, any
element of the corresponding row or column
of~$\hat{\mathcal{S}}(\alpha,\beta;r)$ can be computed as a finite
sum.  In~terms of the unmodified numbers $S_{n,k}(\alpha,\beta;r)$,
the formulas~(\ref{eq:Shatsumformulas}) become
\begin{equation}
  \begin{aligned}
    & 
    {S}_{n,k}(\alpha,\beta;\,r) = \sum_{m=0}^{n-k} \binom{n}{m} S_{n-m,k}(\alpha,\beta;\,0)\,(r)^{\underline m,\alpha},
    \\
    &
    {S}_{n,k}(\alpha,\beta;\,r) = \sum_{m=0}^{n-k} \binom{k+m}m   S_{n,k+m}(\alpha,\beta;\,0)\,(r)^{\underline m,\beta},
  \end{aligned}
\end{equation}
with no restriction on~$\beta$.  As $S_{n,k}(\beta,\beta;\,r) =
\binom{n}k (r)^{\underline{n-k},\beta}$ (see Theorem~\ref{thm:8formulas}
below), these can be seen as specializations of the matrix product
formula~(\ref{eq:Smultrule}).

\section{Stirling-number formulas and combinatorics}
\label{sec:stirlingformulas}

When deriving and interpreting operator ordering identities with
parametric matrix elements $S_{n,k}(\alpha,\beta;r)$ as coefficients,
it is desirable to be able to express them in closed form, and/or
interpret them combinatorially.  For certain choices of~$\alpha,\beta$
and~$r$, both are possible.

The classical Stirling numbers $S_{n,k}(0,1;0)\eqdef\stirsub{n}k$ and
$S_{n,k}(-1,0;0)\eqdef\stircyc{n}k$ have already been mentioned, and
from the recurrence~(\ref{eq:SGKPrec}) it can be proved by induction
that $S_{n,k}(0,1;1)=\stirsub{n+1}{k+1}$ and
$S_{n,k}(-1,0;1)=\stircyc{n+1}{k+1}$.  In~fact when $r\in\mathbb{N}$,
the so-called $r$\nobreakdash-Stirling numbers
$S_{n,k}(0,1;r)\eqdef{\stirsub{n}k}_r$ and
$S_{n,k}(-1,0;r)\eqdef{\stircyc{n}k}_r$ have combinatorial
interpretations~\cite{Broder84,Maier28}.  For instance,
${\stirsub{n}{k}}_r$ counts the number of partitions of an
$(n+\nobreak r)$\nobreakdash-set into $k+\nobreak r$ blocks, such that
$r$~distinguished elements are placed in distinct blocks.  But the
focus in this section is on triples $\alpha,\beta,r$ for which a
closed-form expression for $S_{n,k}(\alpha,\beta;r)$ exists, obviating
the need for an iterative or recursive calculation.  For deriving
additional expressions, homogeneity should be recalled:
$S_{n,k}(\lambda \alpha,\lambda \beta;\lambda r)$ equals
$\lambda^{n-k}S_{n,k}(\alpha,\beta;r)$.

\begin{theorem}
\label{thm:8formulas}
One has the formulas
\begin{itemize}
 \item[{\rm (i)}] $S_{n,k}(1,1;\,0) = \delta_{n,k}$,
 \item[{\rm (ii)}]  $S_{n,k}(-1,1;\,0) = \binom{n}k (k)^{\overline{n-k}}$,
 \item[{\rm (iii)}] $S_{n,k}(1,2;\,0) = 2^{-(n-k)}\,\frac{n!}{k!}\binom{k}{n-k}$,
 \item[{\rm (iv)}] $S_{n,k}(-2,-1;\,0) =2^{-(n-k)}\, \frac{(n)^{\overline{n-k}}\,(k)^{\overline{n-k}}}{(n-k)!}$,
\end{itemize}
which can be extended to\/ $r\in\mathbb{C}${\rm:}
\begin{itemize}
 \item[{\rm (i\,$^\prime$)}] $S_{n,k}(\beta,\beta;\,r) = \binom{n}k (r)^{\underline{n-k},\beta}$,
 \item[{\rm (ii\,$^\prime$)}]  $S_{n,k}(-\beta,\beta;\,r) = \binom{n}k (\beta k+r)^{\overline{n-k},\beta}$,
 \item[{\rm (iii\,$^\prime$)}] $S_{n,k}(1,2;\,r) = 2^{-(n-k)}\,\binom{n}{k} \, {}_2{\hat F}_1\Bigl[
   \begin{smallarray}{c}-(n-k),\:-r \\ -n+2k+1\end{smallarray}
     \Bigm|2\,
     \Bigr]$,
 \item[{\rm (iv\,$^\prime$)}] $S_{n,k}(-2,-1;\,r) = (-2)^{-(n-k)}\,\binom{n}k \, {}_2{\hat
   F}_1\Bigl[
   \begin{smallarray}{c}-(n-k),\:r-1 \\ -2n+k\end{smallarray}
     \Bigm|2\, \Bigr]$.
\end{itemize}
\end{theorem}
\begin{proof}
  Formulas (i)--(iv), (i$^\prime$)--(ii$^\prime$) are verified by
  confirming that they satisfy the recurrence~(\ref{eq:SGKPrec}).
  Formulas equivalent to (iii$^\prime$),(iv$^\prime$) were derived
  in~\cite{Cheon2013} and restated in ${}_2F_1$ terms by the
  author~\cite{Maier28}; they are stated here in~terms of~${}_2\hat
  F_1$.  For an alternative derivation of
  (iii$^\prime$),(iv$^\prime$), see the proof of
  Theorem~\ref{thm:4formulas} below.
\end{proof}

Combinatorial interpretations of these $S_{n,k}(a,b;r)$ can be deduced
from the triangular recurrence~(\ref{eq:SGKPrec}), and a wealth of
information on these and other number triangles can be found in the
OEIS~\cite{OEIS2023}.
The $S_{n,k}(-1,1;0) \eqdef \hat L_{n,k}$ in~(ii) are the unsigned Lah
numbers.  The $S_{n,k}(1,2;0)\eqdef B_{n,k}$ in~(iii) are the Bessel
numbers of the second kind, and the $S_{n,k}(-2,-1;0)\eqdef \hat
b_{n,k}$ in~(iv) are the unsigned Bessel numbers of the first kind.
They appear respectively in the OEIS as \oeisseqnum{A271703},
\oeisseqnum{A122848} (reversed as \oeisseqnum{A100861}),
and~\oeisseqnum{A132062} (reversed as \oeisseqnum{A001498}).  The
formulas (ii),(iii),(iv) are equivalent to those given by
Comtet~\cite[p.~158]{Comtet74}.  When $n\ge1$, the formula~(iv) can be
rewritten in~terms of factorials in the form that appears in
eq.~(\ref{eq:sampleappl}) of the introduction.

These three triangles are the $r=0$ cases of the more general ones in
(ii$^\prime$),(iii$^\prime$),(iv$^\prime$), which have combinatorial
interpretations when $r\in\mathbb{N}$.

\begin{itemize}
\item The $S_{n,k}(-1,1;r)\eqdef \hat L_{n,k}^{(r)}$ in~(ii$^\prime$)
  are what will be called the unsigned $r$\nobreakdash-Lah numbers,
  which when $r$~is even, count the partitions of a set with
  $(n+\nobreak r/2)$ elements into $k+\nobreak r/2$ ordered blocks or
  lists, such that $r/2$~distinguished elements of the set are placed
  in distinct lists~\cite{Nyul2015}. (Unsignedness is indicated by a
  hat; the signed versions are $S_{n,k}(1,-1;-r)\eqdef
  L_{n,k}^{(r)}$.)  Notably, $\hat L_{n,k}^{(2)}= \hat
  L_{n+1,k+1}^{(0)}=\hat L_{n+1,k+1}$.  The triangles $\hat
  L_{n,k}^{(r)}$ with $r=1,2,3,4$ appear separately as
  \oeisseqnum{A021009} (signed), \oeisseqnum{A105278},
  \oeisseqnum{A062139} (signed), \oeisseqnum{A143497}.
\item The $S_{n,k}(1,2;r) \eqdef B_{n,k}^{(r)}$ in~(iii$^\prime$) are
  the $r$\nobreakdash-Bessel numbers of the second kind, which count
  the partitions of an $(n+\nobreak r)$\nobreakdash-set into
  $k+\nobreak r$ blocks of size $1$~or~$2$, such that
  $r$~distinguished elements of the set are placed in distinct
  blocks~\cite{Cheon2013,Jung2018}.  (This implies that when
  $r\in\mathbb{N}$, $S_{n,k}(1,2;r)$~is nonzero if and only if $0\le
  n-\nobreak k \le \left\lfloor\frac{n+r}2\right\rfloor$.)  Notably,
  $B_{n,k}^{(1)}= B_{n+1,k+1}^{(0)}=B_{n+1,k+1}$.  The triangles
  $B_{n,k}^{(r)}$ with $r=1,2$ appear separately as
  \oeisseqnum{A049403}, \oeisseqnum{A330209}.
\item The $S_{n,k}(-2,-1;r) \eqdef \hat b_{n,k}^{(r)}$
  in~(iv$^\prime$) are the unsigned $r$\nobreakdash-Bessel numbers of
  the first kind, which count certain so-called
  G\nobreakdash-permutations of a multiset depending on $n,k$
  and~$r$~\cite{Cheon2013}.  More simply, $\hat b_{n,k}^{(0)}= \hat
  b_{n,k}$ counts $n$\nobreakdash-vertex rooted plane forests composed
  of $k$~rooted plane trees, the vertices being labeled by
  $1,\dots,n$, such that each of the $k$~trees is ``increasing'' in
  that the labels on its vertices increase as one moves away from its
  root~\cite{Bergeron92}.  Notably, $\hat b_{n,k}^{(1)}= \hat
  b_{n+1,k+1}^{(0)}=\hat b_{n+1,k+1}$.  The triangle $\hat
  b_{n,k}^{(r)}$ with $r=1$ appears separately as
  \oeisseqnum{A001497}.
\end{itemize}

Additional expressions are given in the following theorem and
conjecture.

\begin{theorem}
\label{thm:4formulas}
For all\/ $r\in\mathbb{C}$,
\begin{itemize}
\item[{\rm (v)}] $S_{n,k}(-1,2;\,r) = 2^{-(n-k)}\, \binom{n}{k}\,
{}_2{\hat F}_1\Bigl[
   \begin{smallarray}{c}-(n-k),\:-n-r+1 \\ -n+2k+1\end{smallarray}
     \Bigm|2\,
     \Bigr]$,
\item[{\rm (vi)}] $S_{n,k}(-2,1;\,r) = 2^{-(n-k)}\, \binom{n}{k}\,
{}_2{\hat F}_1\Bigl[
   \begin{smallarray}{c}-(n-k),\:-2n-r+1 \\ -2n+k\end{smallarray}
     \Bigm|2\,
     \Bigr]$.
\end{itemize}
\end{theorem}
\begin{proof}
  Each of (v),(vi), and in~fact (iii$^\prime$),(iv$^\prime$), can be
  proved by showing that it yields the correct vertical univariate
  EGF, which was found in the proof of Theorem~\ref{thm:fromvertical}
  to be
  \begin{equation}
    \sum_{n=k}^\infty \hat S_{n,k}(\alpha,\beta;\,r)
    \frac{z^n}{n!}
    =(1+\alpha z)^{r/\alpha} \left[ (1+\alpha z)^{\beta/\alpha} - 1 \right]^k,
  \end{equation}
  where $\hat S_{n,k} = \beta^kk!\,S_{n,k}$.  In doing this, a useful
  lemma is the Srivastava--Singhal generating function for Jacobi
  polynomials~\cite{Srivastava73},
  \begin{equation}
    \sum_{m=0}^\infty
    P_m^{(\alpha-\lambda m,\beta-\mu m)}(x) \,z^m =
    \frac{(1+\xi)^{\alpha+1}\,(1+\eta)^{\beta+1}}{1+\lambda\xi + \mu\eta - (1-\lambda-\mu)\xi\eta}\,,
  \end{equation}
  in which the degree-$m$ Jacobi polynomial $P_m^{(\alpha,\beta)}(x)$
  is defined as usual by
  \begin{equation}
    P_m^{(\alpha,\beta)}(x)=\frac1{m!}\,
    {}_2{\hat F}_1
    \Bigl[
           \begin{smallarray}{c}-m,\:m+\alpha+\beta+1 \\ \alpha+1\end{smallarray}
             \Bigm|\frac{1-x}2\,
             \Bigr],
  \end{equation}
  and $\xi,\eta$ are defined implicitly as functions~of (or~strictly,
  as formal power series~in)~$z$ by
  \begin{subequations}
  \begin{align}
    \xi &= \tfrac12 (x+1)z\, (1+\xi)^{1-\lambda}(1+\eta)^{1-\mu},\\
    \eta &= \tfrac12 (x-1)z\, (1+\xi)^{1-\lambda}(1+\eta)^{1-\mu}.
  \end{align}
  \end{subequations}
  One sets $m=n-k$ and $x=-3$.  To prove (iii$^\prime$),(iv$^\prime$)
  and (v),(vi), one lets the tuple
  $(\lambda,\mu;\allowbreak\alpha,\beta)$ equal $(1,0;\allowbreak
  k,-k-\nobreak r-\nobreak1)$, $(2,-1; \allowbreak
  -k-\nobreak1,\allowbreak k+\nobreak r-\nobreak1)$ and
  $(1,1;\allowbreak k,-2k-\nobreak r)$, $(2,1; -k-\nobreak
  1,-k-\nobreak r+\nobreak1)$, respectively.
\end{proof}

The following formula for $\hat S_{n,k}(-1,2;r)=2^{k}k!\,  
S_{n,k}(-1,2;r)$ was found heuristically and may merit further study.
\begin{conjecture}
  For all $r\in\mathbb{Z}$,
  \begin{displaymath}
    \hat S_{n,k}(-1,2;r)
    =
    \sum_{j=\lfloor(2-r)/2\rfloor}^{\lfloor (n+2-r)/2\rfloor}
    \binom{n-j}{n-k}\,n!\binom{n+1}{2j+r-1}.
  \end{displaymath}
\end{conjecture}
\noindent
This superficially resembles the binomial transform
formula~(\ref{eq:EtoS}) of Theorem~\ref{thm:1}, but the range of the
summation index~$j$ is not $0\le j\le k$, and the upper argument
$n-\nobreak j$ of the first binomial coefficient may be negative.  So
it is not easy to extract the generalized Eulerian numbers
$E_{n,k}(-1,2;r)$ corresponding to the $\hat S_{n,k}(-1,2;r)$ from
this conjectured formula.

\smallskip
The generalized Stirling numbers $S_{n,k}(-1,2;0)$ and
$S_{n,k}(-2,1;0)$ obtained by setting $r=0$ in
Theorem~\ref{thm:4formulas} could be called numbers of the second
kind, and unsigned ones of the first kind, belonging to an as~yet
unnamed combinatorial family.  They appear as \oeisseqnum{A046089}
and~\oeisseqnum{A035342} in the OEIS, where they are interpreted as
counting $n$\nobreakdash-vertex rooted forests composed of $k$~rooted
trees, suitably restricted and labeled.  The more general versions
$S_{n,k}(-1,2;r)$ and $S_{n,k}(-2,1;r)$ seem not to have distinctive
interpretations.  $S_{n,k}(-1,2;1)$ and $S_{n,k}(-2,1;1)$ appear
separately as \oeisseqnum{A321966} and~\oeisseqnum{A265649}, and
notably, one has $S_{n,k}(-1,2;3)=S_{n+1,k+1}(-1,2;0)$ and
$S_{n,k}(-2,1;3)=S_{n+1,k+1}(-2,1;0)$.

It has been observed (see \oeisseqnum{A321966} in the OEIS) that for
all $n\ge0$,
\begin{subequations}
\begin{align}
\label{eq:Hermiteidentity}
& z^n\Hermiteprob_n(z) = \sum_{k=0}^n S_{n,k}(-1,2;\,1)\Hermiteprob_{2k}(z),\\
& \Hermiteprob_{2n}(z) = \sum_{k=0}^n (-1)^{n-k}\,S_{n,k}(-2,1;\,1)\,z^k \Hermiteprob_{k}(z),
\label{eq:Hermiteidentityinverse}
\end{align}
\end{subequations}
the two identities being inverse to each other, as the inverse of the
matrix $\mathcal{S}(-1,2;1)$ is $\mathcal{S}(2,-1;-1)$, and
$S_{n,k}(2,-1;-1)$ equals $(-1)^{n-k} S_{n,k}(-2,1;1)$.  Here
$\Hermiteprob_n$, $n\ge0$, are the probabilists' version of the
physicists' Hermite polynomials~$H_n$, i.e.,
$\Hermiteprob_n(z)=2^{-n/2}\,H_n(z/\sqrt2) =(-1)^n{\rm e}^{z^2/2}
D_z^n {\rm e}^{-z^2/2}$.  

There may be simpler expressions for the connection coefficients in
these identities than are provided by formulas (v) and~(vi) of
Theorem~\ref{thm:4formulas}.  The $S_{n,k}(-1,2;1)$
in~(\ref{eq:Hermiteidentity}) could be computed instead by
\begin{equation}
\label{eq:lesscomb}
S_{n,k}(-1,2;\,1) = 2^{-k}\sum_{j=0}^k (-1)^{k-j}\,
\frac{(2j+1)^{\overline n}}{j!\,(k-j)!},
\end{equation}
which is an application of the summation formula (\ref{eq:Sform}) of
Theorem~\ref{thm:altdef1}.  Also, one should notice that
$\mathcal{S}(-1,2;1) = \mathcal{S}(-1,1;1)\,\mathcal{S}(1,2;0)$, which
is an instance of the matrix product formula~(\ref{eq:Smultrule}).  It
implies that
\begin{equation}
S_{n,k}(-1,2;\,1) = \sum_{j=k}^n \hat L_{n,j}^{(1)}\,B_{j,k}
=
\frac{(n!)^2}{k!} \sum_{j=k}^n \frac{2^{-(j-k)}}{j!\, (n-j)!}\,\binom{k}{j-k}.
\label{eq:seemssimpler}
\end{equation}
Equation~(\ref{eq:seemssimpler}) is more combinatorial
than~(\ref{eq:lesscomb}) in the sense that none of its summands is
negative, and it generalizes to
\begin{equation}
\label{eq:callrewrite}
S_{n,k}(-1,2;\,r) = \sum_{j=k}^n \hat L_{n,j}^{(r)}\,B_{j,k}
=
\frac{n!\,(n+r-1)!}{k!} \sum_{j=k}^n \frac{2^{-(j-k)}}{(j+r-1)!\, (n-j)!}\,\binom{k}{j-k},
\end{equation}
where $r$~is any positive integer.  But
rewriting~(\ref{eq:callrewrite}) in hypergeometric terms and extending
it to all $r\in\mathbb{C}$ would yield an expression comparable to the
${}_2{\hat F}_1$\nobreakdash-based one in part~(v) of
Theorem~\ref{thm:4formulas}.  The numbers $S_{n,k}(-2,1;1)$ in the
inverse identity~(\ref{eq:Hermiteidentityinverse}) will reappear at
the end of the next section.

\smallskip
As regards combinatorial interpretations, it must be said that when
$\alpha,\beta$ and~$r$ are integers, many quite broadly applicable
interpretations of the numbers $S_{n,k}(\alpha,\beta;r)$ are known,
such as ones involving the placement of balls in cells and
compartments, or the drawing of balls from an urn.  (Some are reviewed
in \cite[Section~4.2]{Mansour2016}.)  The ones mentioned in this
section are specific to the choices of $\alpha,\beta,r$ for which a
closed-form expression for $S_{n,k}(\alpha,\beta;r)$ is available,
terminating hypergeometric series being regarded as closed-form
expressions.

\section{Ordering identities (generalized Stirling kind)}
\label{sec:stirlingids}

Operator ordering identities involving the generalized Stirling
numbers of Hsu and Shiue~\cite{Hsu98} will now be derived.  They are of
two types: (1)~the normal ordering of a power~$\Omega^n$ of any
single-annihilator word in~$\adag,a$, and (2)~the expansion
of~$\Omega^n$ in powers~$\Omega'{}^k$, $k=0,\dots,n$, of any other
such word.  Normal orderings are subsumed by expansions, which are
more general.  The respective main theorems are
Theorem~\ref{thm:cor1}, which can be viewed as a special case of a
normal ordering result of M\'endez et~al.~\cite{Mendez2005}, and
Theorem~\ref{thm:powerful}, which is new.

The basic technique used here is a long known and used
one~\cite{Kim2022,Mikhailov85}.  It is the transformation of an
identity in (commutative) polynomial algebra over~$\mathbb{C}$ to one
in operator algebra by replacing an indeterminate~$x$ by~$xD$.  What
results is an identity in~$\tilde W_{\mathbb{C}}$, but if the only
exponents appearing in~it are elements of~$\mathbb{N}$, it can also be
viewed as an identity in~$W_{\mathbb{C}}$.

\begin{theorem}
\label{thm:normord}
For all\/ $L,R\in\mathbb{C}$, one has a normal ordering identity in\/~$\tilde
W^{\rm{bal}}_{\mathbb{C}}${\rm:}
\begin{equation}
\label{eq:normord}
  \left(x^{-e}\right)^n \left(x^L D x^R\right)^n
  = \sum_{k=0}^n S_{n,k}\left(-e,1;\,R\right)
  x^kD^k,
\end{equation}
where\/ $e=L+R-1$.
\end{theorem}
\begin{proof}
  By the definition~(\ref{eq:Sdef1}) of the Hsu--Shiue numbers
  $S_{n,k}(\alpha,\beta;r)$, with $x$ replaced by $xD+r$ and $\beta$
  set equal to\/~$1$,
  \begin{equation}
    (xD+r)^{\underline n,\alpha} = \sum_{k=0}^n
    S_{n,k}(\alpha,1;\,r)\,(xD)^{\underline k}.
  \end{equation}
Applying Boole's identity $(xD)^{\underline k}=x^kD^k$ and the
generalization (\ref{eq:majorrole2a}) given in Lemma~\ref{lem:major}
converts this to
  \begin{equation}
    (x^\alpha)^n(x^{1-\alpha-r}Dx^r)^n
     = \sum_{k=0}^n
    S_{n,k}(\alpha,1;\,r)\,x^kD^k,
  \end{equation}
  and letting $\alpha=1-L-R$, $r=R$ yields~(\ref{eq:normord}).
\end{proof}

\begin{theorem}
For all\/ $L,R\in\mathbb{N}$ with\/ $L+R\ge1$, one has a normal ordering
identity in\/~$W_{\mathbb{C}}${\rm:}
\begin{equation}
\begin{aligned}
\label{eq:untilnow}
  \left(\adag^L a \adag^R\right)^n &= \left(\adag^{e}\right)^n \,\sum_{k=0}^n
  S_{n,k}\left(-e,1;\,R\right) \adag^k a^k \\
  &= \left(\adag^e\right)^n
  \,\sum_{k=0}^n 
  \left[
    \sum_{j=0}^k (-1)^{k-j} \,\frac{(R+j)^{\overline n,e}}{j!(k-j)!}
    \right] \adag^k a^k,
\end{aligned}
\end{equation}
where\/ $e=L+R-1\ge0$ is the excess of the word\/ $\Omega={a^\dag}^L a
{\adag}^R$.
\label{thm:cor1}
\end{theorem}
\begin{proof}
Left-multiply both sides of (\ref{eq:normord}) by
$\left(x^e\right)^n=\left(x^{L+R-1}\right)^n$.  If all powers of $x$
and~$D$ are nonnegative integers, which is the case if
$L,R\in\mathbb{N}$ with $L+\nobreak R\ge\nobreak 1$, substituting
$\adag,a$ for $x,D$ yields an identity in~$W_\mathbb{C}$.  The formula
for $S_{n,k}(-e,1;r)$ is taken from Theorem~\ref{thm:altdef1}, it
being recalled that $\hat S_{n,k}(\alpha,\beta;r) = \beta^k k!
\,S_{n,k}(\alpha,\beta;r)$.
\end{proof}

This is a special case of the abovementioned result of M\'endez
et~al.~\cite{Mendez2005}, who normally-ordered an arbitrary ``boson
string'' of the form $\adag^{r_m}a^{s_m}\dots \adag^{r_1}a^{s_1}$.
(Also see \cite[Section~4.1.2]{Mansour2016}.)  Typical applications
include the following.  Setting $(L,R)=(1,0)$ and~$(0,1)$
in~(\ref{eq:untilnow}) yields the ordering identities shown in
eq.~(\ref{eq:katriel}) of the introduction, which readily generalize
with the aid of Lemma~\ref{lem:major} to
\begin{equation}
\label{eq:katrielplus}
(\adag a)^{\underline n,\alpha} = \sum_{k=0}^n S_{n,k}(\alpha,1;0)\, \adag^k (a)^k,  
\qquad 
(a \adag)^{\underline n,\alpha} = \sum_{k=0}^n S_{n,k}(\alpha,1;1)\, \adag^k (a)^k,
\end{equation}
for all $\alpha\in\mathbb{C}$.  Another application involves the
unsigned $r$\nobreakdash-Lah numbers $\hat
L_{n,k}^{(r)}=S_{n,k}(-1,1;r)$, which by formula~(ii${}^\prime$) of
Theorem~\ref{thm:8formulas} equal $\binom{n}k (k+r)^{\overline{n-k}}$.
Setting $(L,R)$ successively equal to $(0,2)$, $(1,1)$, and~$(2,0)$,
for each of which $e=\nobreak 1$, yields the triple of normal
orderings
\begin{subequations}
  \begin{align}
    \left(\adag^0 a \adag^2\right)^n &= \adag^n\:
    \sum_{k=0}^n \binom{n}k (k+2)^{\overline{n-k}}\,\adag^k a^k,\\
    \left(\adag^1 a \adag^1\right)^n &= \adag^n\:
    \sum_{k=0}^n \binom{n}k (k+1)^{\overline{n-k}}\,\adag^k a^k,\\
    \left(\adag^2 a \adag^0\right)^n &= \adag^n\:
    \sum_{k=0}^n \binom{n}k (k+0)^{\overline{n-k}}\,\adag^k a^k.
  \end{align}
\end{subequations}
An additional application is this.

\begin{corollary}
For all $L,R\in\mathbb{N}$ with $L+R\ge1$, one has a special
normal-ordering identity in\/~$W_{\mathbb{C}}${\rm:}
\begin{equation}
\begin{split}
  &\left(\adag^L a \adag^{R} + \adag^R a \adag^L\right)^n\\
  &\qquad =\left(\adag^{(L+R-1)}\right)^n \,\sum_{k=0}^n \,2^k S_{n,k}\left(2-2L-2R,2;\,L+R\right)
  \adag^k a^k.
\end{split}
\end{equation}
\end{corollary}
\begin{proof}
  If $L+R$ is even, the left-hand side equals $2(\adag)^{(L+R)/2} a
  (\adag)^{(L+R)/2}$ by the identity (\ref{eq:otherpaira})
  in~$W_{\mathbb{C}}$, and the claim follows from
  Theorem~\ref{thm:cor1}, if one also uses the homogeneity property
  $S_{n,k}(\alpha,\beta;r) = 2^{-(n-k)} S_{n,k}(2\alpha,2\beta;2r)$.
  (See eqs.~(\ref{eq:homogeneity}).)  If $L+R$ is odd, one must go
  back to Theorem~\ref{thm:normord}, and apply the counterpart
  identity~(\ref{eq:otherpairb}) in~$\tilde W_{\mathbb{C}}$.
\end{proof}

The following subsumes Theorem~\ref{thm:normord}, because its first
identity, eq.~(\ref{eq:firstmain2a}), reduces to
Theorem~\ref{thm:normord} when $(L',R')=(0,0)$.

\begin{theorem}
For all\/ $L,R,L',R'\in\mathbb{C}$, one has four Stirling-type expansion identities
in\/~$\tilde W_{\mathbb{C}}^{\rm{bal}}${\rm:}
\label{thm:firstmain}
{
\begin{subequations}
\begin{align}
  \left(x^{-e}\right)^n \left(x^L D x^R\right)^n
  &= \sum_{k=0}^n S_{n,k}\left(-e,-e';\,R-R'\right)\left(x^{-e'}\right)^k \left(x^{L'}Dx^{R'}\right)^k   \label{eq:firstmain2a}\\[\jot]
  &= \sum_{k=0}^n S_{n,k}\left(-e,e';\,R+L'-1\right)\left(x^{L'}Dx^{R'}\right)^k\left(x^{-e'}\right)^k,\label{eq:firstmain2b}
\\
\shortintertext{and}
\stepcounter{parentequation}\gdef\theparentequation{\arabic{parentequation}}\setcounter{equation}{0}
  \left(x^L D x^R\right)^n\left(x^{-e}\right)^n 
  &= \sum_{k=0}^n S_{n,k}\left(e,-e';\,1-L-R'\right)\left(x^{-e'}\right)^k \left(x^{L'}Dx^{R'}\right)^k   \label{eq:firstmain2c}\\[\jot]
  &= \sum_{k=0}^n S_{n,k}\left(e,e';\,L'-L\right)\left(x^{L'}Dx^{R'}\right)^k\left(x^{-e'}\right)^k,\label{eq:firstmain2d}
\end{align}
\end{subequations}}where\/ $e=L+R-1$ and\/ $e'=L'+R'-1$.
\end{theorem}
\begin{remarkaftertheorem}
Up~to parameter renamings, the identities
(\ref{eq:firstmain2a}),(\ref{eq:firstmain2d}) are related by the
formal adjoint operation~$*$ on~$\tilde W_{\mathbb{C}}$, which
preserves $\tilde W_{\mathbb{C}}^{\rm bal}$; and so are
(\ref{eq:firstmain2b}),(\ref{eq:firstmain2c}).
\end{remarkaftertheorem}
\begin{proof}
  The identity (\ref{eq:normord}) of Theorem~\ref{thm:normord} is
  based on the matrix $\mathcal{S}(1-\nobreak L-\nobreak R,1;R)$, the
  inverse of which is $\mathcal{S}\bigl(1,1-\nobreak L-\nobreak
  R;-R\bigr)$.  So the inverse identity
  \begin{equation}
    \label{eq:inversenormord}
    x^nD^n = \sum_{k=0}^n S_{n,k}\left(1,1-{L'}-{R'};\,\,{-R'}\right)
    \left(x^{1-{L'}-{R'}}\right)^k \left(x^{L'} D x^{R'}\right)^k
  \end{equation}
  also holds.  Composing (\ref{eq:normord})
  with~(\ref{eq:inversenormord}) yields~(\ref{eq:firstmain2a}),
  because by the matrix product formula~(\ref{eq:Smultrule}), the
  product of $\mathcal{S}(1-\nobreak L-\nobreak R,1;R)$ and
  $\mathcal{S}\bigl(1,\allowbreak 1-\nobreak L'-\nobreak R';-R'\bigr)$
  is $\mathcal{S}(1-\nobreak L-\nobreak R,\allowbreak 1-\nobreak
  L'-\nobreak R';\allowbreak R-\nobreak R')$.  To obtain
  (\ref{eq:firstmain2b}) from~(\ref{eq:firstmain2a}), apply
  (\ref{eq:powerfullemma}) to the right-hand side and perform the
  renamings $L'\leftarrow 1-\nobreak R'$, $R'\leftarrow 1-\nobreak
  L'$.  To obtain (\ref{eq:firstmain2c}) from~(\ref{eq:firstmain2a}),
  apply (\ref{eq:powerfullemma}) to the left-hand side and perform the
  renamings $L\leftarrow 1-\nobreak R$, $R\leftarrow 1-\nobreak L$.
  To obtain (\ref{eq:firstmain2d}) from~(\ref{eq:firstmain2a}), do
  both.
\end{proof}

The following four-identity theorem subsumes Theorem~\ref{thm:cor1},
as the $e=\allowbreak L+\nobreak R-\nobreak 1\ge0$ case of its first
identity reduces to Theorem~\ref{thm:cor1} if one sets
$(L',R')=\allowbreak (0,0)$, so that $\Omega'=\nobreak a$ and
$e'=\nobreak -1$.  That is, it produces a normal ordering of its
left-hand side.  The fourth identity similarly produces an anti-normal
ordering of its left-hand side.

\begin{theorem}
\label{thm:powerful}
For all\/ $L,R,L',R'\in\mathbb{N}$, one has four Stirling-type expansion identities in\/~$W_{\mathbb{C}}${\rm:}
{
\begin{subequations}
\begin{align}
  &\left({a^\dag}^{(E_L-e)}\right)^n \left({a^\dag}^L a {a^\dag}^R\right)^n \label{eq:main1a}\\
  &{}=\sum_{k=0}^n S_{n,k}\left(-e,-e';\,R-R'\right)
  \left({a^\dag}^{E_L}\right)^{n-k}  \left({a^\dag}^{({E_L}-e')}\right)^k  \left({a^\dag}^{L'}a{a^\dag}^{R'}\right)^k   ,\nonumber\\[\jot]
  &\left({a^\dag}^{(E_L-e)}\right)^n \left({a^\dag}^L a {a^\dag}^R\right)^n   \left(\adag^{E_R}\right)^n\label{eq:main1b}\\
  &{}=\sum_{k=0}^n S_{n,k}\left(-e,e';\,R+L'-1\right)
  \left({a^\dag}^{E_L}\right)^{n}   \left({a^\dag}^{L'}a{a^\dag}^{R'}\right)^k    \left({a^\dag}^{(E_R-e')}\right)^k  \left(\adag^{E_R}\right)^{n-k}  ,\nonumber
\\
\shortintertext{and}
\stepcounter{parentequation}\gdef\theparentequation{\arabic{parentequation}}\setcounter{equation}{0}
  &\left({a^\dag}^{E_L}\right)^n \left({a^\dag}^L a {a^\dag}^R\right)^n \left({a^\dag}^{(E_R-e)}\right)^n\label{eq:main2a}\\
  &{}=\sum_{k=0}^n S_{n,k} \left(e,-e';\,1-L-R'\right)
    \left({a^\dag}^{E_L}\right)^{n-k}  \left({a^\dag}^{(E_L-e')}\right)^k   \left({a^\dag}^{L'}a{a^\dag}^{R'}\right)^k   \left(\adag^{E_R}\right)^n ,\nonumber\\[\jot]
  &\left({a^\dag}^L a {a^\dag}^R\right)^n \left({a^\dag}^{(E_R-e)}\right)^n \label{eq:main2b}\\
  &{}=\sum_{k=0}^n S_{n,k}\left(e,e';\,L'-L\right)
   \left({a^\dag}^{L'}a{a^\dag}^{R'}\right)^k  \left(\adag^{(E_R-e')}\right)^k   \left(\adag^{E_R}\right)^{n-k} ,\nonumber
\end{align}
\end{subequations}}where $e=L+R-1$ and $e'=L'+R'-1$ are the excesses of the words
$\Omega={a^\dag}^L a {\adag}^R$ and $\Omega'={a^\dag}^{L'} a
{\adag}^{R'}$\!, and the quantities $E_L=E_L(e,e')$, $E_R=E_R(e,e')$
are as follows:
\begin{center}
\begin{tabular}{c l l}
  Identity & $\hfil E_L$ & $\hfil E_R$ \\
  \hline
  {\rm(\ref{eq:main1a})} & $\max(e,e',0)$ & $\hfil \text{---}\hfil$ \\
  {\rm(\ref{eq:main1b})} & $\max(e,0)$ & $\max(e',0)$ \\
  {\rm(\ref{eq:main2a})} & $\max(e',0)$ & $\max(e,0)$ \\
  {\rm(\ref{eq:main2b})} & $\hfil \text{---}\hfil$ & $\max(e,e',0)$.
\end{tabular}
\end{center}
\end{theorem}

\smallskip
\begin{proof}
  Restrict each of (\ref{eq:firstmain2a}), (\ref{eq:firstmain2b}),
  (\ref{eq:firstmain2c}), (\ref{eq:firstmain2d}) by requiring
  $L,R,\allowbreak L',R'\in\nobreak\mathbb{N}$, and also left-multiply
  by~$(x^{E_L})^n$ and right-multiply by~$(x^{E_R})^n$, where
  $E_L,E_R$ are chosen to make all powers of~$x$ non-negative
  integers; and then substitute $\adag, a$ for~$x,D$.
\end{proof}

The expansion identities of Theorem~\ref{thm:powerful}, being
parametric, yield many individual operator ordering identities
in~$W_{\mathbb{C}}$.  Their coefficients $S_{n,k}$ can be computed
from the summation formula~(\ref{eq:Sform}) in
Theorem~\ref{thm:altdef1} or the triangular
recurrence~(\ref{eq:SGKPrec}) in Theorem~\ref{thm:recurrences}.  But
by using the closed-form expressions for certain~$S_{n,k}$ derived in
Section~\ref{sec:stirlingformulas}, one can make these identities more
explicit, perhaps facilitating new combinatorial interpretations.  The
following discussion focuses on the first and second identities.

It should be noted that the first identity, eq.~(\ref{eq:main1a}),
splits into two or three cases.  Suppose that $e,e'\ge0$ (which
incidentally rules~out the normal ordering case, when $e'=-1$).  Then
if $e\ge e'$, so that $E_L=e$, (\ref{eq:main1a})~says that
\begin{equation}
  \left({a^\dag}^L a {a^\dag}^R\right)^n
  =
  \sum_{k=0}^n S_{n,k}\left(-e,-e';\,R-R'\right)\,
  \adag^{(en-e'k)} \left({a^\dag}^{L'}a{a^\dag}^{R'}\right)^k,
\end{equation}
and if $e\le e'$, so that $E_L=e'$, it says that
\begin{equation}
  {a^\dag}^{(e'-e)n} \left({a^\dag}^L a {a^\dag}^R\right)^n
  =
  \sum_{k=0}^n S_{n,k}\left(-e,-e';\,R-R'\right)
  {a^\dag}^{e'(n-k)}   \left({a^\dag}^{L'}a{a^\dag}^{R'}\right)^k.
\end{equation}
In both cases,
\begin{equation}
  S_{n,k}(-e,-e';\,R-R') =
  \sum_{j=0}^k (-1)^{k-j}\, \frac{(R-R'-e'j)^{\overline n,e}}{j!(k-j)!},
\end{equation}
which comes from Theorem~\ref{thm:altdef1}.  If $e=e'\ge0$, i.e.,
$R+L=R'+L'\ge1$, (\ref{eq:main1a})~becomes an expansion identity with
quite simple coefficients,
\begin{equation}
  \left({a^\dag}^L a {a^\dag}^R\right)^n
  =
  \sum_{k=0}^n \binom{n}k (R-R')^{\overline{n-k},e}\:
  \adag^{e(n-k)} \left({a^\dag}^{L'}a{a^\dag}^{R'}\right)^k,
\end{equation}
where $e=R+L-1=R'+L'-1$.  The coefficients come from
formula~(i$^\prime$) of Theorem~\ref{thm:8formulas}.

Explicit examples of the expansion of a power of a word~$\Omega$ in
powers of another word~$\Omega'$ appeared in eqs.\ (\ref{eq:katriel})
and~(\ref{eq:sampleappl}) of the introduction.  The former is the case
$(L,R;\allowbreak L',R') = (1,0;\allowbreak0,0)$ of the
identity~(\ref{eq:main1a}), and the latter is the case
$(L,R;\allowbreak L',R') = (3,0;\allowbreak2,0)$.  The closed-form
expression for $\hat b_{n,k}=S_{n,k}(-2,-1;0)$ appearing
in~(\ref{eq:sampleappl}) comes not from Theorem~\ref{thm:altdef1} but
from formula~(iv) of Theorem~\ref{thm:8formulas}, having been
rewritten in~terms of ordinary factorials.  The identity inverse
to~(\ref{eq:sampleappl}) is the case $(L,R;\allowbreak
L',R')=(2,0;\allowbreak 3,0)$ of~(\ref{eq:main1a}), which is
\begin{equation}
\begin{aligned}
  \adag^n \left( {a^\dag}^2 a\right)^n  
&=  \sum_{k=0}^n  S_{n,k}(-1,-2;0)\: {a^\dag}^{2(n-k)} \left({a^\dag}^3 a\right)^k\\
&=  \sum_{k=0}^n  (-2)^{-(n-k)}\, \frac{n!}{k!}\,\binom{k}{n-k}  \: {a^\dag}^{2(n-k)} \left({a^\dag}^3 a\right)^k.
\end{aligned}
\label{eq:viewedas}
\end{equation}
Here the coefficient $S_{n,k}(-1,-2;0)$ equals
$(-1)^{n-k}S_{n,k}(1,2;0)$ by homogeneity, and the expression for
$B_{n,k}=S_{n,k}(1,2;0)$ is taken from formula~(iii) of
Theorem~\ref{thm:8formulas}.  The only nonvanishing terms in the sum
are those with $0\le n-\nobreak k\le
\left\lfloor{\frac{n}2}\right\rfloor$.

Equation~(\ref{eq:sampleappl}) has a companion identity, which is the
case $(L,R;\allowbreak L',R')=\allowbreak (2,1;2,0)$ of
eq.~(\ref{eq:main1a}).  It is 
\begin{equation}
\begin{aligned}
  \left( {a^\dag}^2 a\adag\right)^n &= \adag^n\: \sum_{k=0}^n S_{n,k}(-2,-1;1)\:
       {a^\dag}^{(n-k)} \left({a^\dag}^2 a\right)^k\\ &= \adag^n\:\sum_{k=0}^n
       2^{-(n-k)} \frac{(2n-k)!}{k!(n-k)!} \: {a^\dag}^{(n-k)}
       \left({a^\dag}^2 a\right)^k.
\end{aligned}
\end{equation}
The expression used here for the numbers $\hat
b_{n,k}^{(1)}=S_{n,k}(-2,-1;1)$, which are the unsigned
$1$\nobreakdash-Bessel numbers of the first kind, comes from
formula~(iv$^\prime$) of Theorem~\ref{thm:8formulas}.  When $r=1$, the
desingularized hypergeometric series ${}_2\hat F_1(2)$ on which
formula~(iv$^\prime$) is based has zero as an upper parameter.  It
therefore terminates with its first term, which is $(-2n+\nobreak
k)^{\overline{n-k}}$; whence the factor $(2n-\nobreak k)!$.

\smallskip
An interesting example of normal ordering is the following.

\begin{proposition}
  One has the normal-ordering identity
  \begin{displaymath}
    \left(\adag^3 a\right)^n =
    \adag^{2n}\, \sum_{k=0}^n\,
\left\{2^{-(n-k)}\, \binom{n}k  \,
{}_2{\hat F}_1\Bigl[
   \begin{smallarray}{c}-(n-k),\:-2n+1 \\ -2n+k\end{smallarray}
     \Bigm|2\,
     \Bigr]
\right\}
\adag^k a^k
  \end{displaymath}
  in~\/$W_{\mathbb{C}}$, with desingularized hypergeometric series as
  coefficients.
\end{proposition}
\begin{proof}
  This is the case $(L,R;L',R')=(3,0;0,0)$ of eq.~(\ref{eq:main1a}),
  or equivalently the case $(L,R)=(3,0)$ of Theorem~\ref{thm:cor1},
  with the bracketed expression for $S_{n,k}(-2,1;0)$ coming from
  formula~(vi) of Theorem~\ref{thm:4formulas}.
\end{proof}

A seemingly simpler expression for the bracketed coefficient
$S_{n,k}(-2,1;0)$ in this proposition, based on an explicit summation,
is provided by Theorem~\ref{thm:altdef1}: it~is
\begin{equation}
\label{eq:lesscomb2}
S_{n,k}(-2,1;\,0) = \sum_{j=0}^k(-1)^{k-j} \frac{(j)^{\overline n,2}}{j!(k-j)!}.
\end{equation}
A third summation formula for $S_{n,k}(-2,1;0)$, which is more
combinatorial than~(\ref{eq:lesscomb2}), follows from 
$\mathcal{S}(-2,1;0) = \mathcal{S}(-2,-1;0) \,
\mathcal{S}(-1,1;0)$.  It is
\begin{equation}
\begin{aligned}
  S_{n,k}(-2,1;\,0)=
  \sum_{j=k}^n \hat b_{n,j}\hat L_{j,k}
  &=
  \sum_{j=k}^n 2^{-(n-j)}\, \frac{(n)^{\overline{n-j}}(j)^{\overline{n-j}}}{(n-j)!}
  \cdot
  \binom{j}k (k)^{\overline{j-k}}\\
  &= (k)^{\overline{n-k}}\,
  \sum_{j=k}^n 
  2^{-(n-j)}\, \frac{(n)^{\overline{n-j}}}{(n-j)!} \binom{j}k,
\end{aligned}
\end{equation}
because $S_{n,k}(-2,-1;0)=\hat b_{n,k}$ and $S_{n,k}(-1,1;0)=\hat
L_{n,k}$.  Which formula is the most efficient for computing
$S_{n,k}(-2,1;0)$ depends on the (relative) sizes of $k$, $n-\nobreak
k$, and~$n$.

\smallskip
A sample application of the second expansion identity in
Theorem~\ref{thm:powerful}, eq.~(\ref{eq:main1b}), is the following.
As was noted in Section~\ref{sec:stirlingformulas}, the numbers
$S_{n,k}(-2,1;1)$, which are \oeisseqnum{A265649} in the
OEIS~\cite{OEIS2023}, appear as coefficients in an identity satisfied
by the Hermite polynomials $\Hermiteprob_n(z)$.  Setting
$(L,R;\allowbreak L',R')$ in eq.~(\ref{eq:main1b}) successively equal
to $(1,2;\allowbreak 0,2)$, $(2,1;\allowbreak1,1)$, and
$(3,0;\allowbreak 2,0)$, for each of which $E_L=e=\nobreak2$ and
$E_R=e'=\nobreak1$, yields the triple of expansions
\begin{subequations}
  \begin{align}
    \left(\adag^1 a \adag^2\right)^n \adag^n &= \adag^{2n} \:\sum_{k=0}^n S_{n,k}(-2,1;\,1)\, \left(\adag^{0} a \adag^{2}\right)^k \adag^{(n-k)},\\
    \left(\adag^2 a \adag^1\right)^n \adag^n &= \adag^{2n} \:\sum_{k=0}^n S_{n,k}(-2,1;\,1)\, \left(\adag^{1} a \adag^{1}\right)^k \adag^{(n-k)},\\
    \left(\adag^3 a \adag^0\right)^n \adag^n &= \adag^{2n} \:\sum_{k=0}^n S_{n,k}(-2,1;\,1)\, \left(\adag^{2} a \adag^{0}\right)^k \adag^{(n-k)},
  \end{align}
\end{subequations}
where the $S_{n,k}(-2,1;1)$ appear as coefficients.  The connection to
the Hermite polynomial identity~(\ref{eq:Hermiteidentityinverse}), if
any, is unclear.


\section{Eulerian-number formulas and combinatorics}
\label{sec:eulerianformulas}

As with the numbers $S_{n,k}(\alpha,\beta;r)$, it is desirable to be
able to compute the generalized Eulerian numbers
$E_{n,k}(\alpha,\beta;r)$ efficiently, and to interpret them
combinatorially if possible.  They can be computed from the summation
formula (\ref{eq:Eform}) of Theorem~\ref{thm:altdef1}, or the
triangular recurrence~(\ref{eq:ErecGKP}) of
Theorem~\ref{thm:recurrences}.  They can also be computed with the aid
of Theorem~\ref{thm:1}, according to which each row of the matrix
${\mathcal{E}}(\alpha,\beta;r)=[E_{n,k}(\alpha,\beta;r)]$ is the
(inverse) binomial transform of the corresponding row of
$\hat{\mathcal{S}}(\alpha,\beta;r)=[\hat S_{n,k}(\alpha,\beta;r)]$.

Similarly to the $S_{n,k}(\alpha,\beta;r)$, closed-form expressions
are available for certain choices of $\alpha,\beta$ and~$r$.  For
deriving additional closed-form expressions one can use homogeneity:
$E_{n,k}(\lambda \alpha,\lambda \beta;\lambda r)$ equals
$\lambda^{n}E_{n,k}(\alpha,\beta;r)$.  The reflection
property~(\ref{eq:reflectionprop}), according to which
$E_{n,k}(\alpha,\beta;r)$ equals
$E_{n,n-k}(-\alpha,\beta;\beta-\nobreak r)$, can also be used.

\begin{theorem}
\label{thm:Eformulas}
For all\/ $r\in\mathbb{C}$,
\begin{itemize}
 \item[{\rm (i)}] $E_{n,k}(\beta,\beta;\,r) = \binom{n}k
   (r)^{\underline{n-k},\beta}(\beta n-r)^{\underline{k},\beta}$,
 \item[{\rm (ii)}]  $E_{n,k}(-\beta,\beta;\,r) = \binom{n}k (\beta k+r)^{\overline{n-k},\beta}(\beta-r)^{\underline{k},\beta}$.
\end{itemize}
\allowbreak

\noindent
Additionally, one has the formulas
\begin{itemize}
 \item[{\rm (iii)}] $E_{n,k}(-1,2;\,0) =
   \begin{cases}
     \delta_{n,0}, & $k=0$,\\
     n!\,\binom{n+1}{2k-1}, & $k=1,\dots,n$,
   \end{cases}$
 \item[{\rm (iv)}] $E_{n,k}(-1,2;\,1) = n!\,\binom{n+1}{2k}$,
 \item[{\rm (v)}] $E_{n,k}(-1,2;\,2) = n!\,\binom{n+1}{2k+1}$.
\end{itemize}

\noindent
Moreover, $E_{n,k}(-1,2;r)$ for any\/ $r\in\mathbb{N}$ can be computed from the following.
\begin{itemize}
 \item[{\rm (vi)}]
   For all\/ $p\in\mathbb{N}$ and\/ $\zeta\in\{0,1\}$, 
   \begin{align*}
     &E_{n,k}(-1,2;\,2-\zeta+2p)\\
     & {}=n!\,\binom{n+1}{2k+2p+1-\zeta}
     - (-1)^{k+p} \sum_{\ell=0}^{p-1}\, (-1)^{\ell} (2-\zeta+2\ell)^{\overline n}
     \binom{n+1}{k+p-\ell}.
   \end{align*}
\end{itemize}
\end{theorem}
\begin{proof}
Formulas (i)--(v) are verified by confirming that they satisfy the
recurrence~(\ref{eq:ErecGKP}).  Also, by the reflection property
(i),(ii) are equivalent.

The cases $\zeta=1,0$ of formula~(vi) are proved separately, by
induction on~$p$.  Their respective base subcases (with $p=0$) are
formulas (iv) and~(v), and the inductive step in each proof
(an~incrementing of the third parameter by~$\beta=2$) uses the second
difference equation in eq.~(\ref{eq:gaa}), which is a formula for
$\Delta_{r,\beta}E_{n,k}(\alpha,\beta;r)$.
\end{proof}

For computing not individual numbers $E_{n,k}$ but entire row
polynomials, the following simple formula may be of use.  The $n$'th
row polynomial is
\begin{equation}
    E_n(t)= \sum_{k=0}^n E_{n,k}(\alpha,\beta;\,r)\,t^k =(1-t)^{n+1}
    \sum_{j=0}^\infty (\beta j+r)^{\underline n,\alpha} \,t^j.
\end{equation}
(It is not immediately obvious that the last expression, viewed as a
formal power series in~$t$, is in~fact a polynomial in~$t$.)  This
follows from the summation formula~(\ref{eq:Eform}), or alternatively
the EGF~(\ref{eq:EEGF}).  When $(\alpha,\beta;r)=(0,1;0)$, it
specializes to a known formula for the classical Eulerian
polynomials~\cite[p.~245]{Comtet74}, and a version with $\alpha=0$ but
$\beta,r$ arbitrary was obtained previously~\cite{Mezo2016}.

For some choices of the parameters $\alpha,\beta,r$, combinatorial
interpretations of the numbers $E_{n,k}(\alpha,\beta;r)$ are known.
The standard interpretation (when $n>0$) of the classical numbers
$E_{n,k}(0,1;0)$ and $E_{n,k}(0,1;1)$ (i.e., $A_{n,k}$ and
$\eulerian{n}k = A_{n,k+1}$, identical except for a horizontal shift
by unity) has already been mentioned.  The former counts the number of
permutations of the ordered set $\{1,\dots,n\}$ that have exactly
$k$~descents, and the latter is an incremented version.  

The numbers $E_{n,k}(0,2;1) \eqdef {\eulerian{n}{k}}_B$ are less well
known.  They are the MacMahon or type\nobreakdash-$B$ Eulerian
numbers, and count signed permutations of the ordered set
$\{-n,\dots,-1\}\cup\{1,\dots,n\}$ with exactly $k$ ``signed
descents.''  (See~\cite{Bagno2022,Brenti94} and \oeisseqnum{A060187}
in the OEIS\null.)  They are related by the case
$(\alpha,\beta;r)=(0,2;1)$ of the generalized Worpitzky
identity~(\ref{eq:Edef1}) to the type\nobreakdash-$B$ Stirling numbers
$S_{n,k}(0,2;1)\eqdef {\stirsub{n}k}_B$, which can also be interpreted
combinatorially~\cite{Bagno2019}.  

The number triangles $E_{n,k}(0,3;2)$ and $E_{n,k}(0,4;3)$ can be
found in the OEIS as~well, as \oeisseqnum{A225117} and
\oeisseqnum{A225118}.  Like the classical and type\nobreakdash-$B$
Eulerian numbers, they belong to the family of triangles of the form
$E_{n,k}(0,r+\nobreak 1;\nobreak r)$, $r\in\mathbb{N}$, each of which
has an enumerative interpretation~\cite{Wang2023}.  In fact, so does
any triangle $E_{n,k}(0,\beta;r)$ in which $\beta,r\in\mathbb{N}$ and
$r\le\nobreak\beta$ (see~\cite{Ramirez2018}).

An interesting interpretation of $E_{n,k}(-1,\beta;r)/n!$, when
$\beta$~is a positive integer and $r\in\mathbb{N}$ with $r\le\beta$,
should also be mentioned.  It counts the number of words
$s_1s_2,\dots,\nobreak s_n$ with
$s_i\in\{0,\dots,\nobreak\beta-\nobreak 1\}$ that have exactly $k$
descents, provided each word is prefixed with a hidden, ``zeroth''
letter $s_0=\beta\nobreak-r$ (which affects the count of the number of
descents).  Closed-form expressions for $E_{n,k}(-1,2;r)/n!$ with
$r=0,1,2$ are supplied by parts (iii),\allowbreak(iv),(v) of
Theorem~\ref{thm:Eformulas} above.  The latter two triangles appear in
the OEIS as \oeisseqnum{A228955} and \oeisseqnum{A131980}, and the
triangles $E_{n,k}(-1,3;r)/n!$ with $r=1$ and $r=3$ as
\oeisseqnum{A178618} and~\oeisseqnum{A120906}.

Additionally worth mentioning is the recent
interpretation~\cite{Herscovici2020} of the ``degenerate'' Eulerian
numbers of Carlitz~\cite{Carlitz79}, which in present notation would
be denoted by $E_{n,k}(\alpha,1;r)$, in the case when $r=1$.  But
in~general, combinatorial interpretations of the
$E_{n,k}(\alpha,\beta;r)$ have yet to be developed.

\section{Ordering identities (generalized Eulerian kind)}
\label{sec:eulerianids}

Operator ordering identities in~$W_{\mathbb{C}}$ based on the
generalized Eulerian numbers $E_{n,k}$ will now be derived.  Each is
the expansion of the $n$'th power $\Omega^n$ of a single-annihilator
word~$\Omega$ in twistings of the $n$'th power~$\Omega{{}^\prime{}^n}$
of another such word.  More precisely, each provides an expansion of
$\left(\adag^p\right)^n\Omega^n\left(\adag^q\right)^n$, where $p,q$
are sufficiently large, as a combination of terms of the form
$\adag^{p'}\Omega'{}^n\adag^{q'}$, with each term having the same
value of $p'+\nobreak q'$, which is a certain multiple of~$n$.

The main theorem is Theorem~\ref{thm:powerful2}, and like its Stirling
counterpart Theorem~\ref{thm:powerful}, it arises from a simpler
version in~$\tilde W^{\rm bal}_{\mathbb{C}}$ (here,
Theorem~\ref{thm:secondmain}).  The version in~$\tilde W^{\rm
  bal}_{\mathbb{C}}$ comes in~turn from an identity in (commutative)
polynomial algebra, by replacing an indeterminate~$x$ by~$xD$.  Here,
this is the generalized Worpitzky identity~(\ref{eq:Edef1}), which in
the linear-algebraic approach of this paper serves as the definition
of the numbers~$E_{n,k}$.

\begin{theorem}
For all\/ $L,R,L',R'\in\mathbb{C}$, one has four Eulerian-type expansion identities
in\/~$\tilde W_{\mathbb{C}}^{\rm bal}${\rm:}
\label{thm:secondmain}
{
\begin{subequations}
\begin{align}
&
\begin{aligned}
  &n!\,(-e')^n
  \left(x^{-e}\right)^n \left(x^L D x^R\right)^n \\
  &\qquad{}=\sum_{k=0}^n E_{n,k}\left(-e,-e';\,R-R'\right)\left(x^{-e'}\right)^k \left(x^{L'}Dx^{R'}\right)^n \left(x^{-e'}\right)^{n-k},
\end{aligned}
\label{eq:secondmain2a}
\\[\jot]
&
\begin{aligned}
  &n!\,(e')^n
\left(x^{-e}\right)^n \left(x^L D x^R\right)^n \\
  &\qquad{}=\sum_{k=0}^n E_{n,k}\left(-e,e';\,R+L'-1\right) \left(x^{-e'}\right)^{n-k}   \left(x^{L'}Dx^{R'}\right)^n \left(x^{-e'}\right)^k , 
\end{aligned}
\label{eq:secondmain2b}
\\
\shortintertext{and}
\stepcounter{parentequation}\gdef\theparentequation{\arabic{parentequation}}\setcounter{equation}{0}
&
\begin{aligned}
  &n!\,(-e')^n
  \left(x^L D x^R\right)^n\left(x^{-e}\right)^n \\
  &\qquad{}=\sum_{k=0}^n E_{n,k}\left(e,-e';\,1-L-R'\right)\left(x^{-e'}\right)^k \left(x^{L'}Dx^{R'}\right)^n  \left(x^{-e'}\right)^{n-k} ,
\end{aligned}
\label{eq:secondmain2c}
\\[\jot]
&
\begin{aligned}
  &n!\,(e')^n
  \left(x^L D x^R\right)^n\left(x^{-e}\right)^n \\
  &\qquad{}=\sum_{k=0}^n E_{n,k}
  \left(e,e';\,L'-L\right) \left(x^{-e'}\right)^{n-k}   \left(x^{L'}Dx^{R'}\right)^n  \left(x^{-e'}\right)^{k} , 
\end{aligned}
\label{eq:secondmain2d}
\end{align}
\end{subequations}}where\/ $e=L+R-1$ and\/ $e'=L'+R'-1$.
\end{theorem}
\begin{proof}
Substitute $xD+r_1$ for $x$ in the identity~(\ref{eq:Edef1}), with
$r=r_1-\nobreak r_2$.  This yields
\begin{equation}
  n!\,\beta^n(xD+r_1)^{\underline n,\alpha} = \sum_{k=0}^n E_{n,k}(\alpha,\beta;\,r)\,\left[xD+\beta(n-k)+r_2\right]^{\underline n,\beta}.
\end{equation}
  Applying (\ref{eq:majorrole2a}) to both sides converts this to
  \begin{equation}
    \begin{aligned}
      &n!\,\beta^n(x^\alpha)^n\left(x^{1-\alpha-r_1}Dx^{r_1}\right)^n\\
      &\qquad{}= \sum_{k=0}^n E_{n,k}(\alpha,\beta;\,r)
      x^{-r_2}x^{-\beta(n-k)} \left[(x^\beta)^n (x^{1-\beta}D)^n\right]
      x^{\beta(n-k)} x^{r_2} \\
      &\qquad{}= \sum_{k=0}^n E_{n,k}(\alpha,\beta;\,r)
      \left(x^{\beta}\right)^k\left[x^{1-\beta-r_2}Dx^{r_2}\right]^n (x^\beta)^{n-k},
    \end{aligned}
  \end{equation}
  and (\ref{eq:secondmain2a}) follows by letting $L=1-\nobreak
  \alpha-\nobreak r_1$, $R=r_1$, $L'=1-\nobreak \beta-\nobreak r_2$,
  $R'=r_2$.  The expansions
  (\ref{eq:secondmain2a}),(\ref{eq:secondmain2b}) are equivalent by
  homogeneity and the reflection property~(\ref{eq:reflectionprop}),
  as one sees by replacing the summation index~$k$ by~$n-\nobreak k$.
  The expansions (\ref{eq:secondmain2c}),(\ref{eq:secondmain2d}) are
  too, and up~to parameter renamings they are the respective adjoints
  of (\ref{eq:secondmain2b}),(\ref{eq:secondmain2a}).  They also
  follow from (\ref{eq:secondmain2a}),(\ref{eq:secondmain2b}) by
  applying the self-adjoint identity~(\ref{eq:powerfullemma}).
\end{proof}

\begin{theorem}
\label{thm:powerful2}
For all\/ $L,R,L',R'\in\mathbb{N}$, one has four Eulerian-type expansion
identities in\/~$W_{\mathbb{C}}${\rm:}
{
\begin{subequations}
\begin{align}
&
\begin{aligned}
  &n!\, (-e')^n \left({a^\dag}^{(E_L-e)}\right)^n \left({a^\dag}^L a {a^\dag}^R\right)^n    \left(\adag^{E_R}\right)^n \\
  &\qquad{}=\sum_{k=0}^n E_{n,k}\left(-e,-e';\,R-R'\right) \\
  &\qquad\qquad{}\times\left({a^\dag}^{E_L}\right)^{n-k} \left({a^\dag}^{({E_L}-e')}\right)^k   \left({a^\dag}^{L'}a{a^\dag}^{R'}\right)^n \left(\adag^{(E_R-e')}\right)^{n-k} \left(\adag^{E_R}\right)^k,
\end{aligned}
\label{eq:2main1a}
\\[\jot]
&
\begin{aligned}
  &n!\, (e')^n \left({a^\dag}^{(E_L-e)}\right)^n \left({a^\dag}^L a {a^\dag}^R\right)^n   \left(\adag^{E_R}\right)^n \\
  &\qquad{}=\sum_{k=0}^n E_{n,k}\left(-e,e';\,R+L'-1\right)\\
  &\qquad\qquad{}\times\left({a^\dag}^{E_L}\right)^{k} \left(\adag^{(E_L-e')}\right)^{n-k}   \left({a^\dag}^{L'}a{a^\dag}^{R'}\right)^n     \left({a^\dag}^{(E_R-e')}\right)^k  \left(\adag^{E_R}\right)^{n-k},
\end{aligned}
\label{eq:2main1b}
\\
\shortintertext{and}
\stepcounter{parentequation}\gdef\theparentequation{\arabic{parentequation}}\setcounter{equation}{0}
&
\begin{aligned}
  &n!\, (-e')^n \left({a^\dag}^{E_L}\right)^n \left({a^\dag}^L a {a^\dag}^R\right)^n \left({a^\dag}^{(E_R-e)}\right)^n \\
  &\qquad{}=\sum_{k=0}^n E_{n,k}\left(e,-e';\,1-L-R'\right) \\
  &\qquad\qquad{}\times\left({a^\dag}^{E_L}\right)^{n-k} \left({a^\dag}^{(E_L-e')}\right)^k     \left({a^\dag}^{L'}a{a^\dag}^{R'}\right)^n   \left(\adag^{(E_R-e')}\right)^{n-k} \left(\adag^{E_R}\right)^{k},
\end{aligned}
\label{eq:2main2a}
\\[\jot]
&
\begin{aligned}
  &n!\, (e')^n \left(\adag^{E_L}\right)^n  \left({a^\dag}^L a {a^\dag}^R\right)^n \left({a^\dag}^{(E_R-e)}\right)^n \\
  &\qquad{}=\sum_{k=0}^n E_{n,k}\left(e,e';\,L'-L\right)\\
  &\qquad\qquad{}\times\left(\adag^{E_L}\right)^k  \left(\adag^{(E_L-e')}\right)^{n-k}
   \left({a^\dag}^{L'}a{a^\dag}^{R'}\right)^n    \left(\adag^{(E_R-e')}\right)^k \left(\adag^{E_R}\right)^{n-k},
\end{aligned}
\label{eq:2main2b}
\end{align}
\end{subequations}}where\/ $e=L+R-1$ and\/ $e'=L'+R'-1$ are the excesses of the words
$\Omega={a^\dag}^L a {\adag}^R$ and $\Omega'={a^\dag}^{L'} a
{\adag}^{R'}$\!, and the quantities\/ $E_L=E_L(e,e')$, $E_R=E_R(e,e')$
are as follows:
\begin{center}
\begin{tabular}{c l l}
  Identity & $\hfil E_L$ & $\hfil E_R$ \\
  \hline
  {\rm(\ref{eq:2main1a}),(\ref{eq:2main1b})}\: & \,$\max(e,e',0)$ & \:$\max(e',0)$ \\
  {\rm(\ref{eq:2main2a}),(\ref{eq:2main2b})}\: & \,$\max(e',0)$ & \:$\max(e,e',0)$.
\end{tabular}
\end{center}
\end{theorem}
\smallskip
\begin{remarkaftertheorem}
By homogeneity of the $E_{n,k}$ and the reflection
property~(\ref{eq:reflectionprop}), the expansions
(\ref{eq:2main1a}),(\ref{eq:2main1b}) are equivalent, as one sees by
replacing the summation index~$k$ by~$n-\nobreak k$; and so are
(\ref{eq:2main2a}),(\ref{eq:2main2b}).
\end{remarkaftertheorem}
\begin{proof}
  Restrict each of (\ref{eq:secondmain2a}), (\ref{eq:secondmain2b}),
  (\ref{eq:secondmain2c}), (\ref{eq:secondmain2d}) by requiring
  $L,R,\allowbreak L',R'\in\mathbb{N}$, and also left- and
  right-multiply by~$(x^{E_L})^n$, $(x^{E_R})^n$, where $E_L,E_R$ are
  chosen to make all powers of~$x$ non-negative integers; and 
  substitute $\adag, a$ for~$x,D$.
\end{proof}

The identities of Theorem~\ref{thm:powerful2}, being parametric, yield
many individual operator ordering identities in~$W_{\mathbb{C}}$.
Their coefficients $E_{n,k}$ can be computed from the summation
formula~(\ref{eq:Eform}) in Theorem~\ref{thm:altdef1} or the
triangular recurrence~(\ref{eq:SGKPrec}) in
Theorem~\ref{thm:recurrences}.  But by using the closed-form
expressions for certain $E_{n,k}$ in
Section~\ref{sec:eulerianformulas}, one can make these identities more
explicit, perhaps facilitating combinatorial interpretations.  The
following brief discussion focuses on the first
identity~(\ref{eq:2main1a}).

Typical examples of an expansion of a word $\Omega^n$ in twistings of
a word~$\Omega'{}^n$ include the following.  Setting $(L,\nobreak
R;\allowbreak L',\nobreak R')$ in~(\ref{eq:2main1a}) equal to
$(1,\nobreak 0;\allowbreak 0,\nobreak 0)$ and $(0,\nobreak
1;\allowbreak 0,\nobreak 0)$ respectively yields the identities
(\ref{eq:euleriank1}),(\ref{eq:euleriank2}) that appeared in the
introduction.  More generally, when $(L,R)\neq(0,0)$, setting
$(L',R')=(0,0)$, so that $e'=-1$ and $E_L=E_R=0$, yields
\begin{equation}
  n!\,\left(\adag^La\adag^R\right)^n = {\left(\adag^e\right)}^n
\sum_{k=0}^n E_{n,k}(1-L-R,1;\,R)\, \adag^k (a)^n \adag^{(n-k)},
\end{equation}
where $e=L+R-1$.  This representation of any
$\Omega^n=\left(\adag^La\adag^R\right)^n$ with $(L,R)\neq(0,0)$ as a
combination of twistings of ${\Omega'}^n = a^n$ subsumes
(\ref{eq:euleriank1}),(\ref{eq:euleriank2}).

Another example, with explicit coefficients, is
\begin{equation}
\label{eq:last}
  (\adag a \adag)^n \adag^n = \sum_{k=0}^n (-1)^{n-k} \binom{n+1}k
  \adag^{(n-k)} (\adag^2 a)^n \adag^k.
\end{equation}
Here, $\Omega^n= (\adag a \adag)^n$ and $\Omega'{}^n= (\adag^2 a)^n$.
This is the case $(L,R;L',R')=(1,1;2,0)$ of~(\ref{eq:2main1a}), with
$e=e'=1$ and $E_L=E_R=1$, and the expression used for
$E_{n,k}(-1,-1;\allowbreak1)$ comes from formula~(i) of
Theorem~\ref{thm:Eformulas}.

An additional example of the expansion of a word~$\Omega^n$ in
twistings of a word~$\Omega'{}^n$ appeared in
eq.~(\ref{eq:sampleeulerian}) of the introduction.  There, $\Omega^n=
(\adag^2 a)^n$ and $\Omega'{}^n = (\adag^3 a)^n$.
Equation~(\ref{eq:sampleeulerian}) is the case $(L,R;\allowbreak
L',R') = (2,0;\allowbreak3,0)$ of~(\ref{eq:2main1a}), with $e=1$,
$e'=2$, and $E_L=E_R=2$.  The expression for
$E_{n,k}(-1,-2;0)=(-1)^{n} E_{n,k}(1,2;0)$ incorporated
in~(\ref{eq:sampleeulerian}) comes from formula~(v) of
Theorem~\ref{thm:Eformulas}, with the aid of the reflection
identity~(\ref{eq:reflectionprop}).

A final observation is that if a word $\Omega^n=({a^\dag}^L a
{\adag}^R)^n$ or $\Omega'{}^n=({a^\dag}^{L'} a {\adag}^{R'})^n$ in the
Eulerian ordering identities of Theorem~\ref{thm:powerful2} is of the
form $({a^\dag} a {\adag})^n$, it can be rewritten by the
Tait--Toscano--Viskov identity as $\adag^n a^n \adag^n$.  An~instance
occurs in~(\ref{eq:last}), and in consequence, its left-hand side
equals $\adag^n a^n \adag^{2n}$.  A~similar observation applies to the
words $\Omega^n=({a^\dag}^L a {\adag}^R)^n$ and
$\Omega'{}^k=({a^\dag}^{L'}a{\adag}^{R'})^k$ in the Stirling-kind
ordering identities of Theorem~\ref{thm:powerful}.

\section*{Acknowledgements / Author declarations}

This research did not receive any
specific grant from funding agencies in the public, commercial, or
not-for-profit sectors.  Declarations of interest: none.






\end{document}